\documentclass{article}
\usepackage{amsfonts}
\usepackage{amsmath}
\usepackage{amssymb}
\usepackage{multirow}
\usepackage{lscape}
\input xy \xyoption{all}

\usepackage{rotating}
\usepackage{hyperref}
\usepackage{color}

\newtheorem{theorem}{Theorem}[section]

\newtheorem{corollary}[theorem]{Corollary}

\newtheorem{definition}[theorem]{Definition}

\newtheorem{proposition}[theorem]{Proposition}

\newtheorem{remark}[theorem]{Remark}
\newenvironment{proof}[1][Proof]{\noindent\textbf{#1.} }{\ \rule{0.5em}{0.5em}}

\newcommand{\e}{\epsilon}
\newcommand{\f}{\mathfrak}

\newcommand{\dt}{\left.\frac{d}{dt}\right|_{t=0}}

\newcommand{\wnabla}{\widetilde{\nabla}}
\newcommand{\bb}{\mathbb}
\newcommand{\R}{\mathbb{R}}
\newcommand{\C}{\mathbb{C}}

\newcommand{\QK}{\mathcal{QK}}

\font\frd=eufm10 scaled\magstep3
\newcommand{\Cyclic}[2]{\genfrac{}{}{0pt}{0}{#1}{#2}}

\newcommand{\fdS}{\raisebox{-2.0ex}{\mbox{\frd{S}}}}
\newcommand{\SXYZ}{\raisebox{-0.6ex}{\mbox{\scriptsize{$\mathit{XYZ}$}}}}

\newcommand{\Lie}[1]{\textsl{#1}}
\newcommand{\lie}[1]{\mathfrak{#1}}
\newcommand{\Sp}{\Lie{Sp}}
\newcommand{\sP}{\lie{sp}}

\newcommand\T{\rule{0pt}{3.1ex}}
\newcommand\B{\rule[-1.7ex]{0pt}{0pt}}

\title{Homogeneous structures of linear type on $\e$-K\"ahler and $\e$-quaternion K\"ahler manifolds}

\begin{document}

\author{M. Castrill\'on L\'opez \\
\small{ICMAT (CSIC-UAM-UC3M-UCM)}\\
\small{Departamento de Geometr\'\i a y Topolog\'\i a} \\
\small{Facultad de Matem\'aticas, Universidad Complutense de Madrid}\\
\small{28040 Madrid, Spain} \and Ignacio Luj\'an  \\\small{Departamento de Geometr\'\i a y Topolog\'\i a} \\
\small{Facultad de Matem\'aticas, Universidad Complutense de Madrid}\\
\small{28040 Madrid, Spain}}
\date{}

\maketitle

\begin{abstract}We analyze degenerate homogeneous
structures of linear type in the pseudo-K\"ahler and para-K\"ahler
cases. The local form and the holonomy of pseudo-K\"ahler or
para-K\"ahler manifolds admitting such structure are obtained. In
addition the associated homogeneous models are studied exhibiting
their relation with the incompleteness of the metric. The same
questions are tackled in the pseudo-quaternion K\"ahler and
para-quaternion K\"ahler cases.

These results complete the study of homogeneous structures of
linear type in pseudo-K\"ahler, para-K\"ahler, pseudo-quaternion
K\"ahler and para-quaternion K\"ahler cases
\end{abstract}

\renewcommand{\thefootnote}{\fnsymbol{footnote}}
\footnotetext{\emph{MSC2010:} Primary 53C30, Secondary 53C50,
53C55, 53C80.}
\renewcommand{\thefootnote}{\arabic{footnote}}
\renewcommand{\thefootnote}{\fnsymbol{footnote}}
\footnotetext{\emph{Key words and phrases:} homogeneous plane
waves, para-K\"ahler, pseudo-\\ K\"ahler, pseudo-quaternion
K\"ahler, para-quaternion K\"ahler, reductive homogeneous
pseudo-Riemannian spaces.}
\renewcommand{\thefootnote}{\arabic{footnote}}

%\tableofcontents

\section{Introduction}

Ambrose and Singer \cite{AS} generalized Cartan's Theorem on
symmetric spaces characterizing connected, simply-connected and
complete homogeneous Riemannian spaces in terms of a
$(1,2)$-tensor field $S$ called \textit{homogeneous structure
tensor} (or simply \textit{homogeneous structure}) satisfying a
system of geometric PDE's, nowadays called Ambrose-Singer
equations. In \cite{Kir} this result is extended to homogeneous
Riemannian manifolds presenting an additional geometric structure
(such as K\"ahler, quaternion K\"ahler, $G_2$, etc.), and in
\cite{GO} the theory is adapted to metrics with signature.
Homogeneous structures have proved to be one of the most
successful tools in the study of homogeneous spaces, probably due
to the combination of their algebraic and geometric nature. The
first classification of homogeneous structures was provided in
\cite{TV} in the purely Riemannian case, and later in \cite{Fin}
the classification of homogeneous structures for all the possible
holonomy groups in Berger's list is given using a Representation
Theoretical approach. These techniques have been also used  for
metrics with signature (see for instance \cite{BGO}). In many
cases (such as K\"ahler, hyper-K\"ahler, quaternion K\"ahler,
$G_2$ or $\mathrm{Spin}(7)$, as well as in the pseudo-Riemannian
analogs) these classifications contain a class such that the
corresponding pointwise tensor submodule has dimension growing
linearly with the dimension of the manifold. For that reason
homogeneous structures belonging to these classes are called of
\textit{linear type}. The corresponding tensor fields are defined
by a set of vector fields satisfying a system of PDE's equivalent
to Ambrose-Singer equations.

For definite metrics, in the purely Riemannian case so as in the
case of K\"ahler and quaternion K\"ahler manifolds, homogeneous
structures of linear type characterize spaces of negative constant
sectional (resp. holomorphic sectional, quaternionic sectional)
curvature (see \cite{CGS0}, \cite{GMM} and \cite{TV}). When
metrics with signature are studied, the causal character of the
vector fields defining the homogeneous structure tensor needs to
be taken into account. In the purely pseudo-Riemannian case,
non-degenerate structures of linear type (i.e. given by a non null
vector field) characterize spaces of constant sectional curvature
\cite{GO}, while degenerate homogeneous structures of linear type
(i.e. given by a null vector field) characterize singular
scale-invariant plane waves \cite{Mon}. Furthermore, in \cite{Mes}
it is shown that homogeneous structures in the composed class
$\mathcal{S}_1+\mathcal{S}_3$ are related to a larger class of
singular homogeneous plane waves. In \cite{CL} the authors
generalize this result to the pseudo-K\"ahler setting in the
strongly degenerate case, i.e. homogeneous pseudo-K\"ahler
structures of linear type characterized by a null vector field
$\xi$ and vanishing vector field $\zeta$, resulting that the
underlying geometry presents significant analogies with the
geometry of a singular homogenous plane wave. The same problem is
analyzed in the case of pseudo-hyper-K\"ahler and
pseudo-quaternion K\"ahler geometry, finding that a metric
admitting such a structure must be flat.

In the present paper we study degenerate homogeneous structures of
linear type in the pseudo-K\"ahler and the para-K\"ahler settings,
that is, homogeneous pseudo-K\"ahler and para-K\"ahler structures
of linear type defined by a null vector field $\xi$ and an
arbitrary vector field $\zeta$ (see \cite{BGO}). Note that this
includes the strongly degenerate case, so that the results
obtained in this paper generalizes those in \cite{CL}. With these
results, together with those in \cite{LS}, we give a complete
description of the geometry of homogeneous structures of linear
type. Two essentially cases arise in pseudo-K\"ahler and
para-K\"ahler manifolds. On one hand, non-degenerate structures
locally characterize constant curvature spaces. On the other, the
degenerate case (structures studied in \S 3 and \S 4) provides
geometries with an interesting parallelism with homogeneous plane
waves. This is analyzed at the end of the paper. Finally the
pseudo-quaternion and para-quaternion K\"ahler framework does not
provide any geometry other than spaces of constant curvature, a
fact that indicates that the quaternionic realm seems to be too
rigid to contain generalizations of plane waves.

The paper is organized as follows. In Section 2, the general
framework and the notation is settled. Throughout the manuscript,
the notions of pseudo-K\"ahler and para-K\"ahler geometry will be
unified and treated together via the definition of $\e$-K\"ahler
geometry, $\e =\pm 1$. In Section 3, we obtain the curvature and
holonomy of an $\e$-K\"ahler manifold admitting a degenerate
homogeneous $\e$-K\"ahler structure of linear type. In addition we
prove that the vector field $\zeta$ must be a multiple of $\xi$ by
a factor $0$ or $-\e/2$. In Section 4 the local form of a metric
admitting these structures is obtained. The corresponding local
model is studied, focusing in the singular nature of the metric.
In Section 5 the homogeneous model associated to a degenerate
homogeneous $\e$-K\"ahler structure of linear type is computed,
showing that it is (geodesically) incomplete. In section 6 the
same problem in the pseudo-quaternion K\"ahler and para-quaternion
K\"ahler settings is tackled, resulting that the corresponding
metrics must be flat. Finally, section 7 gives a complete view of
the geometry of manifolds endowed with a homogeneous structure of
linear type.

\section{Preliminaries}

We shall combine the treatment of complex geometry and
para-complex geometry by defining $\e=\pm 1$, so that hereafter
$\e$ should be substituted by $-1$ for complex geometry and by $1$
for para-complex geometry (for a survey on para-complex geometry
see for example \cite{CFG}).

\begin{definition}
Let $(M,g)$ be a pseudo-Riemannian manifold of dimension $2n$,

\begin{enumerate}

\item An almost $\epsilon$-Hermitian structure on $(M,g)$ is a
smooth section $J$ of $\f{so}(TM)$ such that $J^2=\epsilon$.

\item $(M,g)$ is called $\epsilon$-K\"ahler if it admits a
parallel almost $\epsilon$-Hermitian structure $J$ with respect to
the Levi-Civita connection.
\end{enumerate}
\end{definition}

The first previous definition implies that the signature of $g$ is
$(2r,2s)$, $r+s=n$, for $\e=-1$, and $(n,n)$ for $\e=1$. The
second definition is equivalent to the holonomy being contained in
$U(r,s)$ for $\e=-1$, and $GL(n,\R)$ for $\e=1$.

Hereafter $(M,g,J)$ is supposed to be a connected $\e$-K\"ahler
manifold of dimension $\mathrm{dim}M\geq 4$.

\begin{definition}
An $\e$-K\"ahler manifold $(M,g,J)$ is called a homogeneous
$\e$-K\"ahler manifold if there is a connected Lie group $G$ of
isometries acting transitively on $M$ and preserving $J$.
$(M,g,J)$ is called a reductive homogeneous $\e$-K\"ahler manifold
if the Lie algebra $\f{g}$ of $G$ can be decomposed as
$\f{g}=\f{h}\oplus\f{m}$ with
$$[\f{h},\f{h}]\subset \f{h},\qquad[\f{h},\f{m}]\subset \f{m}.$$
\end{definition}

Using Kiri\v{c}enko's Theorem  \cite{Kir} (see also \cite{GO}) we
have

\begin{theorem}
Let $(M,g,J)$ be a connected, simply connected and complete
$\e$-K\"ahler manifold. Then the following are equivalent:
\begin{enumerate}
\item $(M,g,J)$ is a reductive homogeneous $\e$-K\"ahler manifold.
\item $(M,g,J)$ admits a linear connection $\wnabla$ such that
\begin{equation}\label{AS+J}
\wnabla g=0, \hspace{1em} \wnabla R=0,\hspace{1em} \wnabla
S=0,\hspace{1em} \wnabla J=0,
\end{equation}
where $S=\nabla-\wnabla$, $\nabla$ is the Levi-Civita connection
of $g$, and $R$ is the curvature tensor of $g$.
\end{enumerate}
\end{theorem}

\begin{definition}
A tensor field $S$ of type $(1,2)$ on an $\e$-K\"ahler manifold
$(M,g,J)$ satisfying (\ref{AS+J}) is called a homogeneous
$\e$-K\"ahler
 structure.
\end{definition}

The classification with respect to the action of the holonomy
group of homogeneous $\e$-K\"ahler structures is carried out in
\cite{BGO} and \cite{GO}, resulting four primitive classes
$\mathcal{K}^{-1}_1,\mathcal{K}^{-1}_2,\mathcal{K}^{-1}_3,\mathcal{K}^{-1}_4$
for $\e=-1$, and eight primitive classes
$\mathcal{K}^{1}_1,\ldots,\mathcal{K}^{1}_8$ for $\e=1$. Among
them, for $\mathcal{K}^{-1}_2$, $\mathcal{K}^{-1}_4$,
$\mathcal{K}^{1}_2$, $\mathcal{K}^{1}_4$, $\mathcal{K}^{1}_6$, and
$\mathcal{K}^{1}_8$ the corresponding pointwise modules have
dimension growing linearly with the dimension of $M$. For this
reason we define

\begin{definition}
A homogeneous $\e$-K\"ahler structure is called of linear type if
it belongs to
\begin{enumerate}
\item $\mathcal{K}^{1}_2\oplus \mathcal{K}^{1}_4$ for $\e=-1$,
\item
$\mathcal{K}^{1}_2\oplus\mathcal{K}^{1}_4\oplus\mathcal{K}^{1}_6\oplus\mathcal{K}^{1}_8$
for $\e=1$.
\end{enumerate}
\end{definition}

The following characterization can be obtained from \cite{BGO} and
\cite{GO}.

\begin{proposition}
A homogeneous $\e$-K\"ahler structure $S$ is of linear type if and
only if
\begin{eqnarray}\label{e-Kahler structure}
S_{X}Y &= & g(X,Y)\xi-g(\xi,Y)X+\e g(X,JY)J\xi-\e
g(\xi,JY)JX\nonumber\\ & - & 2g(\zeta,JX)JY,
\end{eqnarray}
for some vector fields $\xi$ and $\zeta$.
\end{proposition}

Since we are dealing with metrics with signature we further
distinguish the following cases.

\begin{definition}
A homogeneous $\e$-K\"ahler structure of linear type $S$ is called
(see \cite{BGO})
\begin{enumerate}
\item[(i)] non-degenerate if $g(\xi,\xi)\neq 0$,

\item[(ii)] degenerate if $g(\xi,\xi)=0$,

\item[(iii)] strongly degenerate if $g(\xi,\xi)=0$ and $\zeta=0$.
\end{enumerate}
\end{definition}
Case (i) was studied in \cite{LS} and case (iii) was studied in
\cite{CL}. In this paper we concentrate in case (ii).

\section{Degenerate homogeneous $\e$-K\"ahler structures of linear type}

It is a straightforward computation to prove (see \cite{BGO})

\begin{proposition}\label{proposition 2}
A tensor field $S$ on $(M,g,J) $defined by formula (\ref{e-Kahler
structure}) is a homogeneous $\e$-K\"ahler structure if and only
if
$$\wnabla\xi=0,\qquad \wnabla\zeta=0, \qquad \wnabla R=0.$$
where $\wnabla=\nabla-S$.
\end{proposition}

Equation $\wnabla R=0$ reads
\begin{equation}\label{nablaR}\left(\nabla_XR\right)_{YZWU}=-R_{S_XYZWU}-R_{YS_XZWU}-R_{YZS_XWU}-R_{YZWS_XU}\end{equation}
so applying second Bianchi's identity and substituting
(\ref{e-Kahler structure}) we have
\begin{eqnarray}\label{formula sum sym}
0 & = &
\Cyclic{\fdS}{\SXYZ}\left\{2g(X,\xi)R_{YZWU}+g(X,W)R_{YZ\xi
U}+g(X,U)R_{YZW\xi}\right.\nonumber\\
 & & \left.+2\e g(X,JY)R_{J\xi ZWU}+\e
g(X,JW)R_{YZJ\xi U}+\e g(X,JU)R_{YZWJ\xi}\right\}
\end{eqnarray}
Since $g(\xi,\xi)=0$ there exists we an orthonormal basis
$\{e_k\}$ such that $g(e_1,e_1)=1$, $g(e_2,e_2)=-1$, and
$\xi=g(\xi,e_1)(e_1+e_2)$. Whence, contracting the previous
formula with respect to $X$ and $W$ and applying first Bianchi's
identity, we obtain
\begin{eqnarray}\label{formula1}
(2n+2)R_{ZY\xi U} & = &
-2g(Y,\xi)r(Z,U)+2g(Z,\xi)r(Y,U)\nonumber\\
& & -2\e g(Y,JZ)r(J\xi,U)-g(Y,U)r(Z,\xi)\\
& & -\e g(Y,JU)r(Z,J\xi)+g(Z,U)r(Y,\xi)\nonumber\\
& &+\e g(Z,JU)r(Y,J\xi)\nonumber,
\end{eqnarray}
where $r$ denotes the Ricci tensor. With the same orthonormal
basis, contracting the previous expression with respect to $Y$ and
$U$ we arrive to $r(Z,\xi)=(\mathrm{s}/2n)g(Z,\xi)$, where $s$
stands for the scalar curvature. Setting $a=1/(2n+2)$ and
$\nu=\mathrm{s}/2n$, we can write
\begin{equation}\label{formula2}\frac{1}{a}R_{\xi U}=2\theta\wedge r(U)-2\nu\e
\theta(JU)F+\nu U^{\flat}\wedge\theta-\e
\nu(JU)^{\flat}\wedge(\theta\circ J),
\end{equation}
where $F$ denotes the symplectic form associated to $g$ and $J$.
From Bianchi's first identity we have $R_{WUJ\xi\cdot}=R_{\xi
JWU\cdot}-R_{\xi JUW\cdot}$ so we can write (\ref{formula1}) as
\begin{eqnarray}\label{formula3}
0 & = & 2\theta\wedge R_{WU}+W^{\flat}\wedge R_{\xi
U}-U^{\flat}\wedge R_{\xi W}\nonumber\\
& & -2\e F\wedge(R_{\xi JUW}-R_{\xi JWU})\\
& & -\e (JW)^{\flat}\wedge R_{\xi JU}+\e (JU)^{\flat}\wedge R_{\xi
JW}.\nonumber
\end{eqnarray}
Denoting by $\Xi(U)$ the right hand side of (\ref{formula2}) and
substituting in (\ref{formula3}) we obtain
\begin{eqnarray*}
0 & = & \frac{2}{a}\theta\wedge R_{WU}+W^{\flat}\wedge
\Xi(U)-U^{\flat}\wedge \Xi(W)\\
& & -2\e F\wedge \left(i_{W}\Xi(JU)-i_{U}\Xi(JW)\right)\\
& & -\e (JW)^{\flat}\wedge \Xi(JU)+\e (JU)^{\flat}\wedge \Xi(JW).
\end{eqnarray*}
Then, taking $W=\xi$ in the previous formula
$$0=\e(\theta\wedge(\theta\circ J))\wedge(r(JU)-\nu(JU)^{\flat}).$$
Now, since $U$ is arbitrary, denoting $\alpha=r-\nu g$, one has
$$\theta\wedge(\theta\circ J)\wedge\alpha(X)=0,$$
for any vector field $X$. This implies that
$$\alpha=\lambda\theta+\mu\theta\circ J,$$
for some $1$-forms $\lambda$ and $\mu$. Note that since $(M,g,J)$
is $\e$-K\"ahler, $\alpha=r-\nu g$ is symmetric and of type
$(1,1)$. Imposing this to the right hand side of the previous
equality we have that
$$\lambda=f\theta,\qquad \mu=-\e f(\theta\circ J),$$
for some function $f$, so that we obtain
\begin{equation}\label{Ricci}r=\nu g+f\left(\theta\otimes\theta-\e (\theta\circ
J)\otimes(\theta\circ J)\right).\end{equation} Substituting
\eqref{Ricci} in \eqref{formula2} we obtain
\begin{equation}\label{curv 1}\frac{1}{a}R_{\xi
U}=\nu R^0_{\xi U}+P_{\xi U},\end{equation} where
\begin{align*}
R^0_{XYZW} &= g(X,Z)g(Y,W)-g(X,W)g(Y,Z)-\e g(X,JZ)g(Y,JW)\\ &
\hspace{6em}+\e g(X,JW)g(Y,JZ)-2\e g(X,JY)g(Z,JW),
\end{align*}
and
$$P_{\xi U}=-2\e f\theta(JU)\theta\wedge(\theta\circ J).$$

\bigskip

On the other hand, from $\nabla \xi=S\cdot\xi$ and \eqref{e-Kahler
structure}, formula
$$R_{XY}Z=\nabla_{[X,Y]}Z-[\nabla_X,\nabla_Y]Z$$
gives
\begin{equation}\label{curv
2}R_{XY}\xi=-g(\xi,\xi)R^0_{XY}\xi+\Theta^{\zeta}_{XY}\xi=\Theta^{\zeta}_{XY}\xi,\end{equation}
where
\begin{align*}
\Theta^{\zeta}_{XY}\xi =&
-2g(\zeta,JY)g(X,J\xi)\xi+2g(\zeta,Y)g(X,J\xi)J\xi-4g(\zeta,JY)g(X,\xi)J\xi\\
 &+2g(\zeta,JX)g(Y,J\xi)\xi-2g(\zeta,X)g(Y,J\xi)J\xi+4g(\zeta,JX)g(Y,\xi)J\xi\\
 &+4g(\xi,\zeta)g(Y,JX)J\xi-4\e g(\zeta,JY)g(\xi,X)J\xi+4\e
 g(\zeta,JX)g(\zeta,Y)J\xi.
\end{align*}
%\begin{align}
%\Theta^{\zeta}_{XJX}\xi & = -2\e g(\zeta,X)g(X,J\xi)\xi-2\e
%g(\zeta,X)g(X,\xi)J\xi \nonumber\\
%& +2(\zeta,JX)g(\xi,JX)J\xi-2\e g(\zeta,JX)g(X,\xi)\xi\\
%& -4\e g(\xi,\zeta)g(X,X)J\xi -4g(\zeta,X)^2J\xi+4\e
%g(\zeta,JX)^2J\xi \nonumber
%\end{align}
Taking $Y=JX$ and comparing \eqref{curv 1} and \eqref{curv 2}
 one finds that
$$2a\nu g(\xi,JX)=0,\qquad 2a\nu g(\xi,X)=0,$$
for every $X$, so that $\nu =0$. Hence the scalar curvature
vanish. We now choose at every point $p\in M$ a basis
$$\{\xi,J\xi,q_1,Jq_1,X_i,JX_i\}$$ of $T_pM$, where
$g(\xi,q_1)=1$, $g(q_1,q_1)\neq 0$, and $\{X_i,JX_i\}$ is an
orthonormal basis of $\mathrm{span}\{\xi,J\xi,q_1,Jq_1\}^{\bot}$.
Comparing again \eqref{curv 1} and \eqref{curv 2} for $X=\xi$ and
$Y=Jq_1$, and for $X=J\xi$ and $Y=Jq_1$ we obtain that
$g(\zeta,J\xi)=0$ and $g(\zeta,\xi)=0$, so that
$\zeta\in\mathrm{span}\{\xi,J\xi\}^{\bot}$. Taking $X=X_i$ and
$Y=Jq_1$, and $X=JX_i$ and $Y=Jq_1$ we also have $g(\zeta,JX_i)=0$
and $g(\zeta,X_i)=0$ respectively, so that
$\zeta\in\mathrm{span}\{\xi,J\xi\}$. Finally, writing
$\zeta=\lambda\xi+\mu J\xi$ for some functions $\lambda$ and
$\mu$, and taking $X=q_1$ and $Y=Jq_1$ one finds $g(\zeta,Jq_1)=0$
and $2af=-2\e\lambda-4\lambda^2$, so that
$$\zeta=\lambda\xi,\qquad f=-\frac{1}{a}\lambda(\e+2\lambda).$$
Note that equations $\wnabla\xi=0$ and $\wnabla\zeta=0$ imply that
$\lambda$ must be constant. This agrees with the fact that the
Ricci form $$\rho=f\theta\wedge(\theta\circ J)$$ is closed as
$(M,g,J)$ is $\e$-K\"ahler. We have proved

\begin{proposition}\label{proposition zeta}
Let $(M,g,J)$ be a $\e$-K\"ahler manifold admitting a degenerate
$\e$-K\"ahler homogeneous structure of linear type $S$ given by
\eqref{e-Kahler structure}. Then $\zeta=\lambda\xi$ for some
$\lambda\in\R$ and the Ricci curvature is
$$r=-\frac{1}{a}\lambda(\e+2\lambda)\left(\theta\otimes\theta-\e (\theta\circ
J)\otimes(\theta\circ J)\right),$$ where $a=1/(\mathrm{dim}M+2)$
and $\theta=\xi^{\flat}$. In particular the scalar curvature
vanishes.
\end{proposition}

\bigskip

%\subsection{Recurrent formula and curvature computations}

Since $b=0$,  formula \eqref{curv 1} becomes
$$R_{ZY\xi U}=aP_{ZY\xi U}=-2a\e f(\theta\wedge(\theta\circ J)\otimes (\theta\circ J))(Z,Y,U).$$
Looking again to formula \eqref{formula sum sym} we obtain
\begin{eqnarray}\label{formula sum sym 2}
-\Cyclic{\fdS}{\SXYZ}2g(X,\xi)R_{YZWU} & = &
\Cyclic{\fdS}{\SXYZ}2af\left\{(\theta\wedge(\theta \circ
J))\otimes (X^{\flat}\wedge(\theta\circ J))(Y,Z,W,U)\right.\\
 & + & \e(\theta\wedge(\theta \circ
J))\otimes (JX^{\flat}\wedge(\theta))(Y,Z,W,U)\\
& - & \left.2\e g(X,JY)\theta\otimes (\theta\wedge(\theta\circ
J))(Z,W,U)\right\}.
\end{eqnarray}
Substituting this in \eqref{nablaR} and after a quite long
computation we obtain
\begin{equation}\label{eq recursiva}\nabla_XR=4\theta(X)\otimes(R-\frac{1}{2}ag\boxtimes
r)-2a\epsilon \left((X^{\flat}\wedge(\theta\circ
J))\odot\rho+(JX^{\flat}\wedge(\theta))\odot\rho\right),\end{equation}
%$$(\nabla_X R)_{YZWU}=4\theta\otimes\left(R_{YZWU}-afF\odot(\theta\wedge(\theta\circ J))\right).$$
where $\rho$ is the Ricci form and $\boxtimes$ stands for the
$\e$-complex Kulkarni-Nomizu product defined as
\begin{align*}
h\boxtimes k(X_1,X_2,X_3,X_4) & =
h(X_1,X_3)k(X_2,X_4)+h(X_2,X_4)k(X_1,X_3)\\ &-h(X_1,X_4)k(X_2,X_3)-h(X_2,X_3)k(X_1,X_4)\\
 & -\epsilon h(X_1,JX_3)k(X_2,JX_4)-\epsilon
 h(X_2,JX_4)k(X_1,JX_3)\\ & +\epsilon h(X_1,JX_4)k(X_2,JX_3)+\epsilon
 h(X_2,JX_3)k(X_1,JX_4)\\
 & -2\epsilon h(X_1,JX_2)k(X_3,JX_4)-2\epsilon
 h(X_3,JX_4)k(X_1,JX_2),
\end{align*}
for $h$ and $k$ symmetric $(0,2)$-tensors. %Since the Ricci
%form is $\rho=f(\theta\wedge(\theta\circ J))$,
%%denoting
%%$T=R-a\rho\odot F$
% we also have
%\begin{equation}\label{nablaR+Ric}
%\begin{array}{l}
%%\nabla R =4\theta\otimes T,\\
%Ric=f\left(\theta\otimes\theta +(\theta\circ J)\otimes(\theta\circ
%J)\right),\\
%\nabla Ric =2\theta\otimes Ric.
%\end{array}
%\end{equation}

With the help of \eqref{eq recursiva} we now compute some terms of
the curvature tensor of $g$. We again choose a basis
$$\{\xi,J\xi,q_1,Jq_1,X_i,JX_i\}$$
of $T_pM$ for every $p\in M$. Taking the symmetric sum with
respect to $X,Y,Z$ in \eqref{eq recursiva} we have
\begin{align*}
0 & = 4\theta(X)\left(R_{YZWU}-2ag\boxtimes r_{YZWU}\right)\\
  & -2a\e\left((X^{\flat}\wedge(\theta\circ
J))\odot\rho+(JX^{\flat}\wedge(\theta))\odot\rho\right)(Y,Z,W,U)\\
 & -2a\e\left((Y^{\flat}\wedge(\theta\circ
J))\odot\rho+(JY^{\flat}\wedge(\theta))\odot\rho\right)(Z,X,W,U)\\
 & -2a\e\left((Z^{\flat}\wedge(\theta\circ
J))\odot\rho+(JZ^{\flat}\wedge(\theta))\odot\rho\right)(X,Y,W,U).
\end{align*}
Setting $Y,Z\in\mathrm{span}\left\{\xi,J\xi\right\}^{\bot}$ we
obtain
\begin{equation}\label{R_xy}R_{YZWU}=-8a\e g(Y,JZ)\rho(W,U),\qquad Y,Z\in\mathrm{span}\{\xi,J\xi\}^{\bot}\end{equation}
for every $W,U$. On the other hand setting $X=q_1$, $Y=Jq_1$ and
$Z\in\mathrm{span}{X_i,JX_i}$ we find
$$R_{YZWU}=af\left(g(Z,W)\theta(JU)-g(Z,U)\theta(JW)-g(Z,JW)\theta(U)+g(Z,JU)\theta(W)\right),$$
for every $W,U$, so that
\begin{align}\label{R_q1x}R_{q_1ZWU} &=af\left(g(JZ,U)\theta(JW)-g(JZ,W)\theta(JU)+\e g(Z,U)\theta(W)\right.\nonumber\\
& \left.-\e g(Z,W)\theta(U)\right)\end{align} for
$Z\in\mathrm{span}\{X_i,JX_i\}$ and all $W,U$.

\begin{proposition}\label{proposition Ricci-flat}
$(M,g,J)$ is Ricci-flat.
\end{proposition}

\begin{proof}
Let $g(q_1,q_1)=b$ and suppose for simplicity that $b>0$ (the case
$b<0$ is analogous). Denoting $q_2=Jq_1$, we choose an orthonormal
basis
$$\left\{\sqrt{b}\left(\xi-\frac{q_1}{b}\right),\sqrt{b}\left(J\xi-\frac{q_2}{b}\right),\frac{q_1}{\sqrt{b}},\frac{q_2}{\sqrt{b}},X_i,JX_i\right\}$$
of $T_pM$ for every $p\in M$, which has signature
$(-1,\e,1,-\e,\varepsilon^i,-\epsilon\varepsilon^i)$ where
$g(X_i,X_i)=\varepsilon^i\in\{\pm 1\}$. We compute the Ricci
curvature by contracting the curvature tensor with respect to this
orthonormal basis and using \eqref{R_xy} and \eqref{R_q1x}:
\begin{align*}
r(W,U) &=
-R\left(W,\sqrt{b}\left(\xi-\frac{q_1}{b}\right),U,\sqrt{b}\left(\xi-\frac{q_1}{b}\right)\right)\\
 & +\e R\left(W,\sqrt{b}\left(\xi-\frac{q_2}{b}\right),U,\sqrt{b}\left(\xi-\frac{q_2}{b}\right)\right)\\
 & +
 R\left(W,\frac{q_1}{\sqrt{b}},U,\frac{q_1}{\sqrt{b}}\right)-\e R\left(W,\frac{q_2}{\sqrt{b}},U,\frac{q_2}{\sqrt{b}}\right)\\
 & +\varepsilon^iR(W,X_i,U,X_i)-\e\varepsilon^iR(W,JX_i,U,JX_i)\\
 & = 4af\left(\theta\otimes\theta-\e (\theta\circ
J)\otimes(\theta\circ J)\right)\\
& + 2af\e\sum_i\varepsilon^i\left(\theta\otimes\theta-\e
(\theta\circ J)\otimes(\theta\circ J)\right)\\
& = (4a+2a\e\sum_i\varepsilon^i)r(W,U).
\end{align*}
We deduce that if $r(W,U)\neq 0$ then
$4a+2a\e\sum_i\varepsilon^i=1$, therefore
$$\mathrm{dim}M+2=4+2a\e\sum_i\varepsilon^i,$$
whence
$$\mathrm{dim}M=2+2a\e\sum_i\varepsilon^i<\mathrm{dim}M.$$
Since this is impossible we conclude that $r=0$.
\end{proof}

\begin{corollary}\label{corollary lambda values}
The only possible values for $\lambda$ are $\lambda=0$ and
$\lambda=-\frac{\e}{2}$.
\end{corollary}
In the next section we shall study the cases $\lambda=0$ and
$\lambda=-\frac{\e}{2}$ separately.

\begin{proposition}\label{proposition holonomy}
The curvature tensor of $g$ is given by
$$R=k (\theta\wedge (\theta\circ J))\otimes (\theta\wedge (\theta\circ J)),$$
for some function $k$. Moreover, if $k\neq 0$ the holonomy algebra
of $g$ is given by
$$\f{hol}\cong\bb{R}\left(\begin{array}{ccl}i_{\e} & i_{\e} & 0\\ -i_{\e} & -i_{\e} & 0\\ 0 & 0 & 0_n\end{array}\right),$$
where $i_{\e}$ is the $\e$-complex imaginary unit, which is a one
dimensional subalgebra of $\f{su}(1,1)\subset\f{su}(r,s)$,
$r+s=n+2$, for $\e=-1$, and $\f{sl}(2,\R)\subset\f{sl}(n+2,\R)$
for $\e=1$.

\end{proposition}

\begin{proof}
Since $(M,g,J)$ is Ricci-flat, we have that $f=0$, so that
\eqref{eq recursiva} becomes
$$\nabla R=4\theta\otimes R.$$
Taking symmetric sum in the previous formula and applying second
Bianchi identity we have that $\theta\wedge R_{WU}=0$ for every
$W,U$. But from the $\e$-K\"ahler symmetries of $R$ we also have
$(\theta\circ J)\wedge R_{WU}=0$. These force the curvature to be
of the form
$$R=k (\theta\wedge (\theta\circ J))\otimes (\theta\wedge (\theta\circ J)),$$
for some function $k$.

On the other hand, since $(M,g,J)$ is real analytic, the
infinitesimal holonomy algebra coincides with the holonomy algebra
(see \cite{KN}). Recall that the infinitesimal holonomy algebra at
$p\in M$ is defined as
$\f{hol}'=\bigcup_{k=0}^{\infty}\f{m}_k$, where
$$\f{m}_0=\mathrm{span}\{R_{XY}/X,Y\in T_pM\}$$
and
$$\f{m}_k=\mathrm{span}\left\{\f{m}_{k-1}\cup\{(\nabla_{Z_k}\ldots\nabla_{Z_1}R)_{XY}/Z_1,\ldots,Z_k,X,Y\in T_pM\}\right\}.$$
As a simple computation shows one has
$$\nabla \theta=\theta\otimes\theta+(2\lambda+\e)(\theta\circ J)\otimes (\theta\circ J).$$
It is easy to see that this together with the recurrent formula
$\nabla R=4\theta\otimes R$ imply that
$\f{m}_0=m_1=\ldots=\f{m}_k$ for every $k\in \mathbb{N}$, so that
$\f{hol}'=\f{m}_0$. Now, since $R=k (\theta\wedge (\theta\circ
J))\otimes (\theta\wedge (\theta\circ J))$ the space $\f{m}_0$ is
the one dimensional space generated by the endomorphism
$$\begin{array}{rrcl}
A: & T_pM & \to & T_pM\\
   & \xi,J\xi & \mapsto & 0\\
   & q_1 & \mapsto & J\xi\\
   & q_2 & \mapsto & \e\xi\\
   & X_i,JX_i & \mapsto & 0.
\end{array}$$
This endomorphism is expressed as
$$\frac{1}{b}\left(\begin{array}{ccl}i_{\e} & i_{\e} & 0\\ -i_{\e} & -i_{\e} & 0\\ 0 & 0 & 0_n\end{array}\right)$$
with respect to the $\e$-complex orthonormal basis
$$\left\{\frac{1}{\sqrt{|b|}}(q_1+\e i_{\e}q_2),\left(\frac{1}{\sqrt{|b|}}q_1-s\sqrt{|b|}\xi\right)+
\e
i_{\e}\left(\frac{1}{\sqrt{|b|}}q_2-s\sqrt{|b|}J\xi\right),X_i+\e
i_{\e}JX_i\right\},$$ where $g(q_1,q_1)=b$ and $s$ is the sign of
$b$.
\end{proof}

%\subsection{Some global properties}

As a consequence of Proposition \ref{proposition holonomy} we have
that for all the values $\e=\pm 1$ and $\lambda=0,-\frac{\e}{2}$,
$(M,g,J)$ is an \textit{Osserman manifold} with a $2$-step
nilpotent Jacobi operator. It is also easy to see that $(M,g,J)$
is VSI (vanishing scalar invariants). Finally, it is worth noting
that if $(M,g,J)$ is connected and simply-connected, then it is
the product of a $2n$-dimensional $\e$-complex flat and totally
geodesic manifold which can be thought of as an $\e$-complex
wavefront, and a $4$-dimensional Walker $\e$-K\"ahler manifold
with a parallel null $\e$-complex vector field, which can be think
of as the $\e$-complex time and direction of propagation of the
wave.

%Finally we point out some global properties of $(M,g,J)$ shared
%for all the values $\e=\pm 1$ and $\lambda=0,-\frac{\e}{2}$.
%
%\begin{enumerate}
%
%\item By Proposition \ref{proposition holonomy}, if $(M,g,J)$ is
%connected and simply-connected, then it is the product of a
%$2n$-dimensional $\e$-complex flat and totally geodesic manifold
%which can be thought of as an $\e$-complex wavefront, and a
%$4$-dimensional Walker $\e$-K\"ahler manifold with a parallel null
%$\e$-complex vector field, which can be think of as the
%$\e$-complex time and direction of propagation of the wave.
%
%\item $(M,g,J)$ is VSI (vanishing scalar invariants): since the
%curvature only involves $\theta$ and $\theta\circ J$, it is easy
%to see that the inverse metric forces all possible scalar
%invariants to vanish.
%
%\item $(M,g,J)$ is an \textit{Osserman manifold} with a $2$-step
%nilpotent Jacobi operator: the Jacobi operators, that is
%$\mathrm{J}(X): Y\mapsto R_{YX}X$, of the elements of the basis
%$\{\xi,J\xi,q_1,q_2,X_a,JX_a\}$ are
%$$\mathrm{J}(q_1)=\left(\begin{array}{ccccl}0 & 0 & 0 & 0 & 0\\ 0 & 0 & 0 & k\e & 0\\ 0 & 0 & 0 & 0 & 0\\
%0 & 0 & 0 & 0 & 0\\ 0 & 0 & 0 & 0 & 0_{2n}\end{array}\right),$$
%$$\mathrm{J}(q_2)\left(\begin{array}{ccccl}0 & 0 & 0 & -k & 0\\ 0 & 0 & 0 & 0 & 0\\ 0 & 0 & 0 & 0 & 0\\
%0 & 0 & 0 & 0 & 0\\ 0 & 0 & 0 & 0 & 0_{2n}\end{array}\right),$$
%$$\mathrm{J}(\xi)=0,\qquad \mathrm{J}(J\xi)=0,\qquad \mathrm{J}(X_a)=0,\qquad \mathrm{J}(JX_a)=0.$$
%
%
%\end{enumerate}

\section{Local form of the metrics}

In previous sections (Prop. \ref{proposition zeta} and Prop.
\ref{proposition Ricci-flat}) we have seen that an $\e$-K\"ahler
manifold $(M,g,J)$ admitting a degenerate homogeneous
$\e$-K\"ahler structure of linear type satisfies
$\zeta=\lambda\xi$ for some constant $\lambda\in \R$ and is
Ricci-flat. As stated in Corollary \ref{corollary lambda values}
this implies that the only possible values for $\lambda$ are
$\lambda=0$ and $\lambda=-\frac{\e}{2}$. Hereafter $M$ is supposed
to be non-flat and of dimension $2n+4$.

\subsection{$\lambda=-\frac{\e}{2}$}\label{subsection local lambda=0.5}

We shall obtain the local form of the metric in the case
$\lambda=-\frac{\e}{2}$ for both values of $\e$ simultaneously.

Substituting the value $\lambda=-\frac{\e}{2}$ in \eqref{e-Kahler
structure} we have obtain
\begin{equation*}
S_{X}Y = g(X,Y)\xi-g(\xi,Y)X+\e g(X,JY)J\xi-\e g(\xi,JY)JX + \e
g(\xi,JX)JY.
\end{equation*}
The condition $\wnabla \xi=0$ then implies
$$\nabla\xi=\theta\otimes\xi,$$
which gives
\begin{align*}\nabla \theta & =\theta\otimes\theta,\nonumber\\
\nabla \theta\circ J & =\theta\otimes(\theta\circ J).
\end{align*}
In particular $d\theta=0$ so that fixing a point $p\in M$ there is
a neighborhood $\mathcal{U}$ and a function $v:\mathcal{U}\to \R$
such that $\theta=dv$. We consider
$$w_1=e^{-v},$$
whence $dw_1=-e^{-v}=-w_1\theta$. We now consider
$$dw_1\circ J=-w_1(\theta\circ J).$$
Differentiating we obtain
$$d(dw_1\circ J)=-dw_1\wedge(\theta\circ J)-w_1d(\theta\circ J)=w_1\theta\wedge(\theta\circ J)-w_1\theta\wedge(\theta\circ J)=0.$$
Therefore, reducing $\mathcal{U}$ if necessary, there is a
function $w_2:\mathcal{U}\to \R$ such that $dw_2=\e dw_1\circ J$.
Let $\C^{\e}$ denote the complex or the para-complex numbers
depending on the value of $\e$ and $i_{\e}$ the complex or
para-complex imaginary unit accordingly. We consider the function
$w=w_1+i_{\e}w_2$. Then $dw=dw_1+\e i_{\e}(dw_1\circ J)$, so that
$dw$ is of type $(1,0)$ with respect to $J$ and
$w:\mathcal{U}\to\C^{\e}$ is $\e$-holomorphic. In addition it is a
straightforward computation to see that
$$\nabla dw=-dw_1\otimes\theta-w_1\nabla\theta-i_{\e}dw_1\otimes(\theta\circ J)-i_{\e}w_1\nabla(\theta\circ J)=0,$$
i.e., $dw$ is a nowhere vanishing parallel $1$-form.

The function $w:\mathcal{U}\to\C^{\e}$ defines a foliation of
$\mathcal{U}$ by $\e$-complex hypersurfaces
$\mathcal{H}_{\tau}=w^{-1}(\tau)$, $\tau\in\C^{\e}$ (where
$w^{-1}(\tau)$ is non empty). Note that since the tangent space to
$\mathcal{H}_{\tau}$ is given by the kernel of $dw$, the
hypersurfaces $\mathcal{H}_{\tau}$ are tangent to the distribution
$\mathrm{span}\{\xi,J\xi\}^{\bot}$. We consider the vector field
$$Z=\mathrm{grad}(w_1)=dw_1^{\sharp}.$$
It is easy to see that by construction
$$JZ=-\e\mathrm{grad}(w_2).$$
These vector fields are written as
\begin{align*}
Z &= -w_1\xi,\\
JZ &= -w_1J\xi,
\end{align*}
so that
$$\nabla
Z=-dw_1\otimes\xi-w_1\nabla\xi=w_1\theta\otimes\xi-w_1\theta\otimes\xi=0,$$
and thus also $\nabla JZ=0$. This implies in particular that $Z$
and  $JZ$ are commuting $\e$-holomorphic Killing vector fields.

We now look at the holonomy of $g$ at $p$, which was computed in
Proposition \ref{proposition holonomy}. Using the same notation as
in the previous section we denote
$E=\mathrm{span}\{\xi,J\xi,q_1,q_2\}\subset T_pM$. This subspace
is invariant under the holonomy action and so is $E^{\bot}$. In
fact, the holonomy action on $E^{\bot}$ is trivial. This implies
that taking an orthonormal basis $\{(X_a)|_p,(JX_a)|_p\}$,
$a=1,\ldots,n$, of $E^{\bot}$ we can extend them to vector fields
$\{X_a,JX_a\}$
 using the parallel transport, i.e., we have an
orthonormal reference $\{X_a,JX_a\}$ on $\mathcal{U}$ such that
$\nabla X_a=0=\nabla JX_a$, $a=1,\ldots,n$. In particular they are
commuting $\e$-holomorphic Killing vector fields. In addition, let
$\gamma$ be any smooth curve on $\mathcal{U}$, we have
$$\dt \left(dw(X_a)_{\gamma(t)}\right)=\left(\nabla_{\dot{\gamma}(t)}dw\right)(X_a)+dw\left(\nabla_{\dot{\gamma}(t)}X_a\right)=0,$$
whence the function $dw(X_a)$ is constant along $\gamma$ and takes
the value $0$ at $p$. This implies that $X_a$ and thus $JX_a$ is
tangent to the foliation $\mathcal{H}_{\tau}$. Finally note that
since they are parallel $X_a$ and $JX_a$ commute with $Z$ and
$JZ$.

We have thus constructed a set of commuting para-holomorphic
Killing vector fields $\{Z,JZ,X_a,JX_a\}$ tangent to
$\mathcal{H}_{\tau}$. Therefore, reducing $\mathcal{U}$ if
necessary, we can take $\e$-complex coordinates $\{w,z,z^a\}$ on
$U$ such that $\partial_z=\frac{1}{2}(Z+\e i_{\e}JZ)$,
$\partial_{z^a}=\frac{1}{2}(X_a+\e i_{\e}JX_a)$. Note that since
the distributions $\mathrm{span}\{\partial_w,\partial_z\}$ and
$\mathrm{span}\{\partial_{z^a},a=1,\ldots,n\}$ are invariant by
holonomy, the vector fields $X_a$ and $JX_a$ are orthogonal to
$\mathrm{span}\{\partial_w,\partial_z\}$. Writing
$z=z^1+i_{\e}z^2$, $z^a=x^a+i_{\e}y^a$ and $w=w^1+i_{\e}w^2$, and
rearranging the coordinates as $\{z^1,z^2,w^1,w^2,x^a,y^a\}$, we
have that the metric with respect to these coordinates is
$$g=\begin{pmatrix}
0 & 0 & 1 & 0 & 0 & \ldots & 0\\
0 & 0 & 0 & -\e & 0 & \ldots & 0\\
1 & 0 & b & 0 & 0 & \ldots & 0\\
0 & -\e & 0 & -\e b & 0 & \ldots & 0\\
0 & 0 & 0 & 0 & & & \\
\vdots & \vdots & \vdots & \vdots & & \Sigma & \\
0 & 0 & 0 & 0 & & &
\end{pmatrix},$$
for some function $b$, where
$$\Sigma=\mathrm{diag}(\begin{pmatrix}\varepsilon^a & 0\\0 &-\e\varepsilon^a\end{pmatrix},a=1,\ldots,n),$$
with $\varepsilon^a=g(X_a,X_a)\in\{\pm 1\}$. With respect to this
coordinates the $\e$-complex structure reads
$$J=\begin{pmatrix}
0 & \e &        &  &  \\
1 & 0 &        &  &  \\
  &   & \ddots &  &  \\
  &   &        & 0 & \e\\
  &   &        & 1 & 0
\end{pmatrix}.$$
Imposing that $\partial_{z^1}$, $\partial_{z^2}$, $\partial_{x^a}$
and $\partial_{y}^a$ are parallel it is easy to see that $b$ does
not depend on $z^1,z^2,x^a,y^a$.

Finally, computing the curvature tensor with respect to this
coordinates we obtain
$$R=\frac{1}{2}\Delta^{\e}b(dw^1\wedge dw^2)\otimes(dw^1\wedge dw^2),$$
where
$$\Delta^{\e}=-\e\frac{\partial^2}{\partial (w^1)^2}+\frac{\partial^2}{\partial (w^2)^2}.$$
Denoting $F=\Delta^{\e}b$ and taking into account that $dw^1$ and
$dw^2$ are parallel, we have that
$$\nabla R=\frac{1}{2}dF\otimes(dw^1\wedge dw^2)\otimes(dw^1\wedge dw^2).$$
Recall that the recurrent formula \eqref{eq recursiva} together
with the Ricci-flatness of $g$ give that
$$\nabla R=4\theta \otimes R.$$
Comparing these two formulas for $\nabla R$ we have that
$$dF=4F\theta,$$
where $\theta$ can be written as
$$\theta=-\frac{1}{w^1}dw^1.$$
Note that by construction $w^1\neq 0$. The system of partial
differential equations is thus
\begin{align*}
\frac{\partial F}{\partial w^1} &= -\frac{4}{w_1}F\\
\frac{\partial F}{\partial w^2} &= 0,
\end{align*}
which has solution
$$F=\frac{R_0}{(w^1)^4},$$
for some constant $R_0\in\R$. We have thus proved

\begin{proposition}\label{proposition local 2}
Let $(M,g,J)$ be an $\e$-K\"ahler manifold of dimension $2n+4$,
$n\geq 0$, admitting a degenerate homogeneous $\e$-K\"ahler
structure of linear type $S$ with $\zeta=-\frac{\e}{2}\xi$. Then
each $p\in M$ has a neighborhood $e$-holomorphically isometric to
an open subset of $(\C^{\e})^{n+2}$ with the $\e$-K\"ahler metric
\begin{equation}\label{local metric 2}
g  = dw^1dz^1-\e dw^2dz^2+b(dw^1dw^1-\e
dw^2dw^2)+\sum_{a=1}^n\varepsilon^a(dx^adx^a-\e dy^ady^a),
\end{equation}
where $\varepsilon^a=\pm 1$, and the function $b$ only depends on
the coordinates $\{w^1,w^2\}$ and satisfies
\begin{equation*}\Delta^{\e}b=\frac{R_0}{(w^1)^4}\end{equation*} for
$R_0\in\R-\{0\}$.
\end{proposition}

\subsection{$\lambda=0$}\label{subsection local lambda=0}

The case $\e=-1$ and $\lambda=0$ was studied in \cite{CL}, where
the local form and some properties of the metric $g$ and the
complex structure $J$ were obtained. We reproduce below the main
lines of the proof in \cite{CL}, now for both values of $\e$
simultaneously and putting special attention to the formulas which
depend on $\e$.

Substituting the value $\lambda=0$ in \eqref{e-Kahler structure}
we have that the homogeneous structure $S$ takes the form
$$S_XY=g(X,Y)\xi-g(\xi,Y)X+\e g(X,JY)J\xi-\e g(\xi,JY)JX.$$
In analogy with the complex case we shall call $S$
\textit{strongly degenerate}.
%The condition $\wnabla \xi=0$ then
%implies
%$$\nabla\xi=\theta\otimes \xi-\e (\theta\circ J)\otimes J\xi,$$
%which gives \begin{align}\label{nabla theta e=1
%lambda=0}\nabla\theta &=\theta\otimes \theta+\e(\theta\circ
%J)\otimes(\theta\circ J),\nonumber\\
%\nabla(\theta\circ J) &=\theta\otimes(\theta\circ J)+(\theta\circ
%J)\otimes \theta.
%\end{align}
We consider the form $\alpha=\theta+i\e(\theta\circ J)$, which is
of type $(1,0)$ with respect to the $\e$-complex structure $J$
(see \cite{GO}). As a straightforward computation shows,
$\nabla\alpha=\alpha\otimes\alpha$ so that $d\alpha=0$. This
implies in particular that $\alpha$ is an $\e$-holomorphic
$1$-form. Fixing a point $p\in M$, by the closeness of $\alpha$
there is a neighborhood $\mathcal{U}$ of $p$ and an
$\e$-holomorphic function $v:\mathcal{U}\to \C^{\e}$ such that
$\alpha=dv$. We consider the $\e$-holomorphic function
$$w=e^{-v},$$ where the
the exponential must read $e^{x+i_{\e}y}=e^x(\cos y+i_{\e}\sin y)$
for $\e=-1$ and $e^{x+i_{\e}y}=e^x(\cosh y+i_{\e}\sinh y)$ for
$\e=1$. Differentiating we obtain that $\nabla dw=0$ so that
%$$dw=-e^{-v}dv=-w\alpha,$$
%so that
%$$\nabla dw=-dw\otimes\alpha-w\nabla\alpha=w\alpha\otimes\alpha-w\alpha\otimes\alpha=0,$$
$dw$ is a nowhere vanishing parallel $\e$-holomorphic $1$-form on
$\mathcal{U}$. The function $w:\mathcal{U}\to\C^{\e}$ defines a
foliation of $\mathcal{U}$ by $\e$-complex hypersurfaces
$\mathcal{H}_{\tau}=w^{-1}(\tau)$, $\tau\in\C^{\e}$ (if
$w^{-1}(\tau)$ is non empty).

It is easy to adapt the construction made in \cite{CL} for both
values of $\e$ simultaneously in order to find a set of
coordinates $\{z^1,z^2,w^1,w^2,x^a,y^a\}$ with respect to which
the metric takes the form
$$g  = dw^1dz^1-\e dw^2dz^2+b(dw^1dw^1-\e
dw^2dw^2)+\sum_{a=1}^n\varepsilon^a(dx^adx^a-\e dy^ady^a),$$ and
$J$ is the standard $\e$-complex structure of $\R^{2n+4}$, where
$\varepsilon^a=\pm 1$ and the function $b$ only depends on the
coordinates $\{w^1,w^2\}$. In addition, as a simple computation
shows,
$$R=\frac{1}{2}\Delta^{\e}b(dw^1\wedge dw^2)\otimes(dw^1\wedge dw^2),$$
and
$$\theta=\frac{-1}{(w^1)^2-\e(w^2)^2}(w^1dw^1-\e w_2dw^2).$$
Finally, imposing the equation $\nabla R=4\theta\otimes R$ and
denoting $F=\Delta^{\e}b$ we obtain the system of partial
differential equations
\begin{align*}
\frac{\partial F}{\partial w^1} &= \frac{-4w^1}{(w^1)^2-\e(w^2)^2}F\\
\frac{\partial F}{\partial w^2} &= \frac{4\e
w^2}{(w^1)^2-\e(w^2)^2}F.
\end{align*}
which has solution
$$F=\frac{R_0}{((w^1)^2-\e(w^2)^2)^2},$$
for some constant $R_0\in \R$. Note that since $w=e^{-v}$ we
always have $(w^1)^2-\e(w^2)^2\neq 0$. We have thus proved

%Finally, computing the curvature tensor with respect to this
%coordinates we obtain
%$$R=\frac{1}{2}\Delta^{\e}b(dw^1\wedge dw^2)\otimes(dw^1\wedge dw^2).$$
%Denoting $F=\Delta^{\e}b$ and taking into account that $dw^1$ and
%$dw^2$ are parallel, we have that
%$$\nabla R=\frac{1}{2}dF\otimes(dw^1\wedge dw^2)\otimes(dw^1\wedge dw^2).$$
%Recall that the recurrent formula \eqref{eq recursiva} together
%with the Ricci-flatness of $g$ give that
%$$\nabla R=4\theta \otimes R.$$
%Comparing these two formulas for $\nabla R$ we have that
%$$dF=4F\theta,$$
%where $\theta$ can be written as
%$$\theta=\frac{-1}{(w^1)^2-\e(w^2)^2}(w^1dw^1-\e w_2dw^2).$$
%Note that since $w=e^{-v}$ we always have $(w^1)^2-\e(w^2)^2\neq
%0$. The system of partial differential equations is thus
%\begin{align*}
%\frac{\partial F}{\partial w^1} &= \frac{-4w^1}{(w^1)^2-\e(w^2)^2}F\\
%\frac{\partial F}{\partial w^2} &= \frac{4\e
%w^2}{(w^1)^2-\e(w^2)^2}F,
%\end{align*}
%which has solution
%$$F=\frac{R_0}{((w^1)^2-\e(w^2)^2)^2},$$
%for some constant $R_0\in \R$. We have thus proved

\begin{proposition}\label{proposition local 1}
Let $(M,g,J)$ be an $\e$-K\"ahler manifold of dimension $2n+4$,
$n\geq 0$, admitting a strongly degenerate homogeneous
$\e$-K\"ahler structure  of linear type $S$. Then each $p\in M$
has a neighborhood $\e$-holomorphically isometric to an open
subset of $(\C^{\e})^{n+2}$ with the $\e$-K\"ahler metric
\begin{equation}\label{local metric 1}
g  = dw^1dz^1-\e dw^2dz^2+b(dw^1dw^1-\e
dw^2dw^2)+\sum_{a=1}^n\varepsilon^a(dx^adx^a-\e dy^ady^a),
\end{equation}
where $\varepsilon^a=\pm 1$, and the function $b$ only depends on
the coordinates $\{w^1,w^2\}$ and satisfies
\begin{equation*}\Delta^{\e}b=\frac{R_0}{((w^1)^2-\e(w^2)^2)^2} \end{equation*} for
$R_0\in\R-\{0\}$.
\end{proposition}

\subsection{The manifold $((\C^{\e})^{n+2},g)$}\label{section
manifolds}

Propositions \ref{proposition local 1} and \ref{proposition local
2} give the local forms \eqref{local metric 1} and \eqref{local
metric 2} of the metric of a manifold with a degenerate
homogeneous $\e$-K\"ahler structure of linear type. This motivates
the study of the space $(\C^{\e})^{2+n}$ endowed with this
particular $\e$-K\"ahler metric, which can thus be understood as
the simplest instance of this type of manifolds. In particular,
the goal of this section is to study the singular nature of this
spaces, and their analogies with homogeneous plane waves. We shall
restricts ourselves to the Lorentz $\e$-K\"ahler case, i.e.,
metrics of index $2$. Throughout this section $\| w \|_{\lambda}$
must be understood as
\begin{equation}
\label{norma}
\| w\|_{\lambda}^2=\left\{\begin{array}{l}w_1^2-\e
w_2^2\hspace{2em} \text{for} \quad \lambda=0,\\ w_1^2
\hspace{21mm} \text{for} \quad \lambda=-\e /2.\end{array}\right.
%
%\| w\|_{\lambda}^2&=&(w_1)^2-\e (w_2)^2,\qquad \text{for} \quad \lambda=0,\\
%\| w\|_{\lambda}^2&=&(w_1)^2,\qquad \qquad \text{for} \quad \lambda=-\e /2.\\
\end{equation}
In addition $\Delta^{\e}$ shall stand for the differential
operator
$$\Delta^{\e}=-\e\frac{\partial^2}{\partial w_1^2}+\frac{\partial^2}{\partial w_2^2}.$$

We thus consider the manifold $(\C^{\e})^{2+n}=(\R^{2n+4},J_0)$,
where $J_0$ is the standard $\e$-complex structure, with real
coordinates $\{z^1,z^2,w^1,w^2,x^a,y^a\}$, endowed with the metric
\begin{equation}\label{metric general form}
g = dw^1dz^1-\e dw^2dz^2+b(dw^1dw^1-\e
dw^2dw^2)+\sum_{a=1}^n(dx^adx^a-\e dy^ady^a),
\end{equation}
where $b$ is a function of the variables $(w^1,w^2)$ satisfying
\begin{equation}\label{differential equation b general form }
\Delta^{\e}b=\frac{R_0}{\| w\|_{\lambda}^4}, \qquad
R_0\in\R-\{0\}.
\end{equation}
As computed before the curvature $(1,3)$-tensor field of $g$ is
\begin{eqnarray*}
R &=&\frac{1}{2}\frac{R_0}{\| w\|_{\lambda}^4}\left((dw^1\wedge
dw^2)\otimes(dw^1\otimes \partial_{z^2})\right.\\
& & \hspace{27mm}\left.+\e(dw^1\wedge dw^2)\otimes(dw^2\otimes
\partial_{z^1})\right).
\end{eqnarray*}
As $R_0\neq 0$, it exhibits a singular behavior at
$$\mathcal{S}=\{\| w\|_{\lambda}=0\}.$$
This set can be understood as a singularity of $g$ in the
cosmological sense: the geodesic deviation equation is governed by
the components of the curvature tensor $R_{w^1w^2w^i}^{z^j}$,
$i,j=1,2$, making the tidal forces infinite at $\mathcal{S}$.
Indeed,
%The geodesic equation for the variables $(w^1,w^2)$ are
%\[
%\ddot{w}^1=0,\qquad \ddot{w}^2=0.
%\]
%Then, geodesics with initial value
%$$(w^1(0),w^2(0))=(a,b),\qquad (\dot{w}^1(0),\dot{w}^2(0))=(c,cb/a),$$ with $b\in\bb{R}$,
%$a,c\in\bb{R}-\{0\}$, are
%$$w^1(t)=ct+a,\qquad w^2(t)=c\frac{b}{a}t+b,$$
%and geodesics with initial value $$(w^1(0),w^2(0))=(0,b), \qquad
%(\dot{w}^1(0),\dot{w}^2(0))=(0,d),$$ with $b,d\in\bb{R}-\{0\}$,
%are
%$$w^1(t)=0,\qquad w^2(t)=dt+b.$$ These geodesics reach the singular
%set $\mathcal{S}$ in finite time $t=-\frac{a}{c}$ and
%$t=-\frac{b}{d}$ respectively. Hence
%$((\C^{\e})^{2+n}-\mathcal{S},g)$ is not geodesically complete.
%\end{enumerate}
%
%There exists the possibility that this singularity is due to a bad
%choice of coordinates so that we could embed
%$((\C^{\e})^{2+n}-\mathcal{S},g)$ in a complete manifold
%$(\hat{M},\hat{g})$ with $\hat{g}$ of class $\mathcal{C}^2$ . To
%see that this is actually not possible,
we can compute a component of the curvature tensor with respect to
an orthonormal parallel frame along a geodesic reaching the
singular set in finite time, and see that it is singular (see
\cite{Sen}). Let $\gamma$ be the geodesic with initial value
$\gamma(0)=(0,0,1,0,\ldots,0)$ and
$\dot{\gamma}=(0,0,-1,0,\ldots,0)$. It is easy to see that this
geodesic is of the form
$$\gamma(t)=(z^1(t),z^2(t),1-t,0,x^a(t),y^a(t))$$
for some functions $z^1(t),z^2(t),x^a(t),y^a(t)$, $a=1,\ldots,n$,
and reaches the singular set $\mathcal{S}$ at $t=1$. Let
$$E(t)=W^1(t)\partial_{w^1}+W^2(t)\partial_{w^2}+Z^1(t)\partial_{z^1}+Z^2(t)\partial_{z^2}+X^a(t)\partial_{x^a}+Y^a(t)\partial_{y^a}$$
be a vector field along $\gamma$. $E$ is parallel if the following
equations hold:
\begin{align*}
0=\dot{W}^1,& \qquad 0=\dot{W}^2,\\
0=\dot{Z}^1-W^1\Gamma_{w^1w^1}^{z^1}-W^2\Gamma_{w^1w^2}^{z^1},&
\qquad
0=\dot{Z}^2-W^1\Gamma_{w^1w^1}^{z^2}-W^2\Gamma_{w^1w^2}^{z^2},\\
0=\dot{X}^a,& \qquad 0=\dot{Y}^a.
\end{align*}
We can thus obtain an orthonormal parallel frame
$\{E_1(t),\ldots,E_{4+2n}(t)\}$ with $E_1(t)$ and $E_2(t)$ of the
form
\begin{align*}
E_1(t)=& \frac{1}{\sqrt{|b(0)|}}\partial_{w^1}+Z_1^1(t)\partial_{z^1}+Z_1^2(t)\partial_{z^2}+X_1^a\partial_{x^a}+Y_1^a\partial_{y^a},\\
E_2(t)=&
\frac{1}{\sqrt{|b(0)|}}\partial_{w^2}+Z_2^1(t)\partial_{z^1}+Z_2^2(t)\partial_{z^2}+X_2^a\partial_{x^a}+Y_2^a\partial_{y^a},
\end{align*}
where $E_1(0)=\frac{1}{\sqrt{|b(0)|}}\partial_{w^1}$,
$E_2(0)=\frac{1}{\sqrt{|b(0)|}}\partial_{w^2}$, and $b(0)=b(0,0)$.
The value of the curvature tensor applied to $E_1(t),E_2(t)$ is
$$R_{E_1(t)E_2(t)E_1(t)E_2(t)}=\frac{R_0}{2b(0)^2}\frac{1}{||w(t)||_\lambda ^4}=\frac{R_0}{2b(0)^2}\frac{1}{(1-t)^4},$$
which is singular at $t=1$.

\bigskip

\begin{center}
\begin{tabular}{|c|c|c|}
\hline
Singular set $\mathcal{S}$ & $\lambda=0$ & $\lambda=-\frac{\e}{2}$\\[0.5ex]
\hline $\e=-1$ & & \\ &
\def\svgwidth{3cm}%% Creator: Inkscape 0.48.3.1, www.inkscape.org
%% Accompanies image file 'dibujo3.pdf' (pdf, eps, ps)
\begingroup%
  \makeatletter%
  \providecommand\color[2][]{%
    \errmessage{(Inkscape) Color is used for the text in Inkscape, but the package 'color.sty' is not loaded}%
    \renewcommand\color[2][]{}%
  }%
  \providecommand\transparent[1]{%
    \errmessage{(Inkscape) Transparency is used (non-zero) for the text in Inkscape, but the package 'transparent.sty' is not loaded}%
    \renewcommand\transparent[1]{}%
  }%
  \providecommand\rotatebox[2]{#2}%
  \ifx\svgwidth\undefined%
    \setlength{\unitlength}{370.028125bp}%
    \ifx\svgscale\undefined%
      \relax%
    \else%
      \setlength{\unitlength}{\unitlength * \real{\svgscale}}%
    \fi%
  \else%
    \setlength{\unitlength}{\svgwidth}%
  \fi%
  \global\let\svgwidth\undefined%
  \global\let\svgscale\undefined%
  \makeatother%
  \begin{picture}(1,0.95127904)%
    \put(0,0){\includegraphics[width=\unitlength]{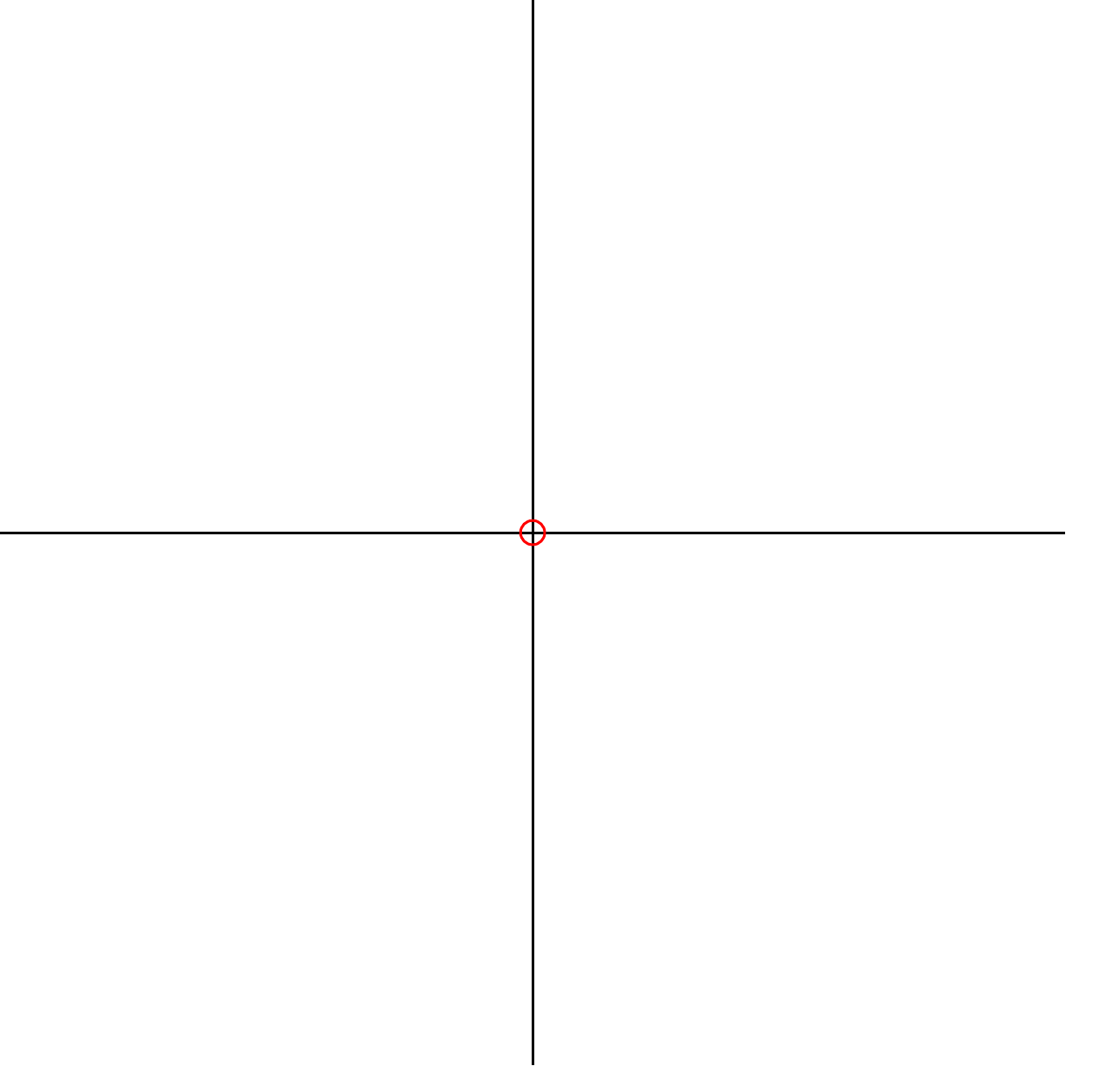}}%
    \put(0.85398914,0.45188526){\color[rgb]{0,0,0}\makebox(0,0)[lt]{\begin{minipage}{0.46699153\unitlength}\raggedright $w^1$\end{minipage}}}%
    \put(0.33294766,0.95130675){\color[rgb]{0,0,0}\makebox(0,0)[lt]{\begin{minipage}{0.25295375\unitlength}\raggedright $w^2$\end{minipage}}}%
  \end{picture}%
\endgroup%
 &
\def\svgwidth{3cm}%% Creator: Inkscape 0.48.3.1, www.inkscape.org
%% Accompanies image file 'dibujo2.pdf' (pdf, eps, ps)
\begingroup%
  \makeatletter%
  \providecommand\color[2][]{%
    \errmessage{(Inkscape) Color is used for the text in Inkscape, but the package 'color.sty' is not loaded}%
    \renewcommand\color[2][]{}%
  }%
  \providecommand\transparent[1]{%
    \errmessage{(Inkscape) Transparency is used (non-zero) for the text in Inkscape, but the package 'transparent.sty' is not loaded}%
    \renewcommand\transparent[1]{}%
  }%
  \providecommand\rotatebox[2]{#2}%
  \ifx\svgwidth\undefined%
    \setlength{\unitlength}{370.028125bp}%
    \ifx\svgscale\undefined%
      \relax%
    \else%
      \setlength{\unitlength}{\unitlength * \real{\svgscale}}%
    \fi%
  \else%
    \setlength{\unitlength}{\svgwidth}%
  \fi%
  \global\let\svgwidth\undefined%
  \global\let\svgscale\undefined%
  \makeatother%
  \begin{picture}(1,0.95127904)%
    \put(0,0){\includegraphics[width=\unitlength]{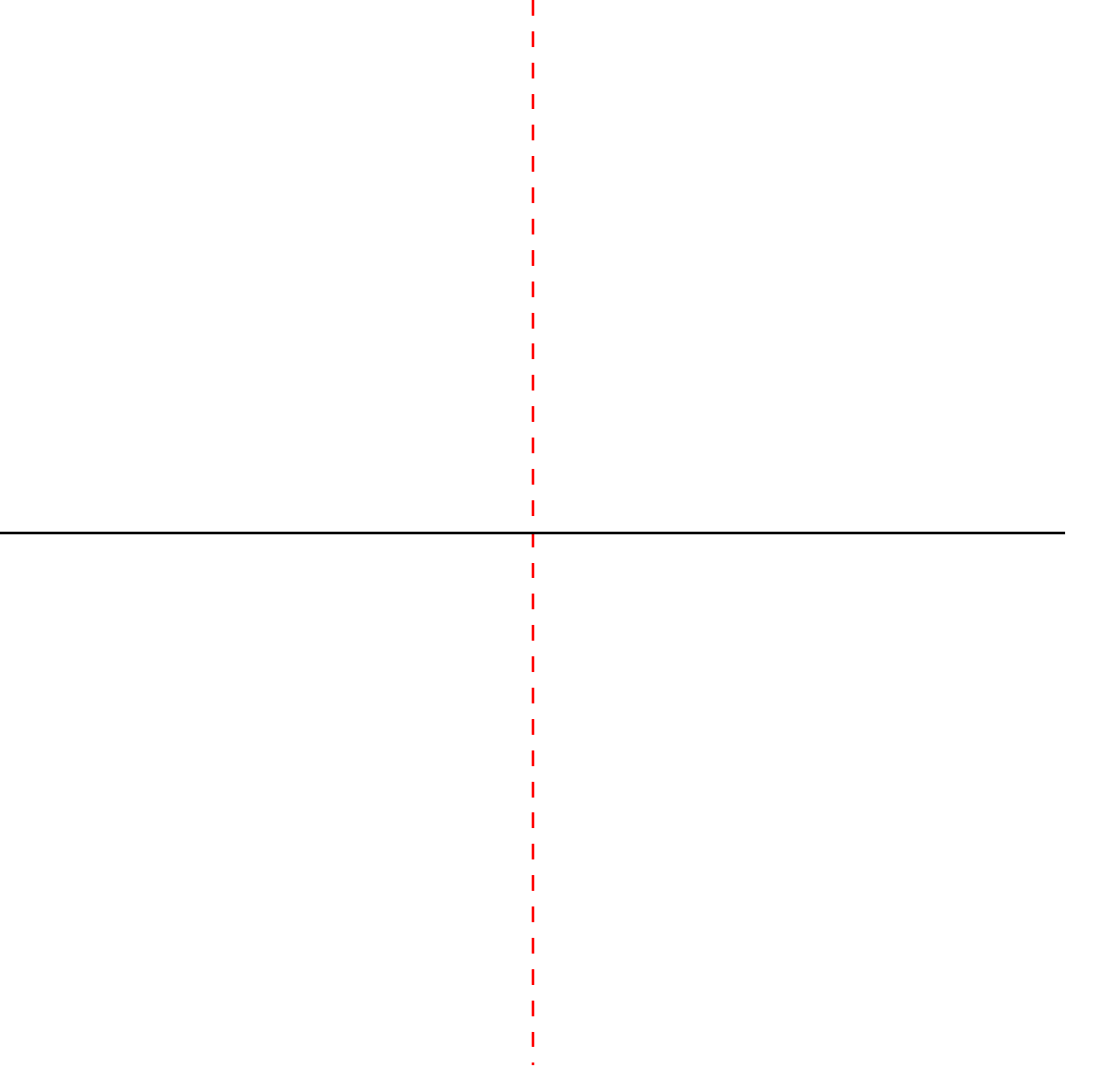}}%
    \put(0.85398914,0.45188526){\color[rgb]{0,0,0}\makebox(0,0)[lt]{\begin{minipage}{0.46699153\unitlength}\raggedright $w^1$\end{minipage}}}%
    \put(0.34808165,0.95130675){\color[rgb]{0,0,0}\makebox(0,0)[lt]{\begin{minipage}{0.25295375\unitlength}\raggedright $w^2$\end{minipage}}}%
  \end{picture}%
\endgroup%
\\[0.5ex]
 &  \footnotesize{$\mathcal{S}: (w^1)^2+(w^2)^2=0$} &  \footnotesize{$\mathcal{S}: w^1=0$}\\[0.5ex]
\hline $\e=1$ & & \\ &
\def\svgwidth{3cm}%% Creator: Inkscape 0.48.3.1, www.inkscape.org
%% Accompanies image file 'dibujo.pdf' (pdf, eps, ps)
\begingroup%
  \makeatletter%
  \providecommand\color[2][]{%
    \errmessage{(Inkscape) Color is used for the text in Inkscape, but the package 'color.sty' is not loaded}%
    \renewcommand\color[2][]{}%
  }%
  \providecommand\transparent[1]{%
    \errmessage{(Inkscape) Transparency is used (non-zero) for the text in Inkscape, but the package 'transparent.sty' is not loaded}%
    \renewcommand\transparent[1]{}%
  }%
  \providecommand\rotatebox[2]{#2}%
  \ifx\svgwidth\undefined%
    \setlength{\unitlength}{370.303125bp}%
    \ifx\svgscale\undefined%
      \relax%
    \else%
      \setlength{\unitlength}{\unitlength * \real{\svgscale}}%
    \fi%
  \else%
    \setlength{\unitlength}{\svgwidth}%
  \fi%
  \global\let\svgwidth\undefined%
  \global\let\svgscale\undefined%
  \makeatother%
  \begin{picture}(1,0.95138274)%
    \put(0,0){\includegraphics[width=\unitlength]{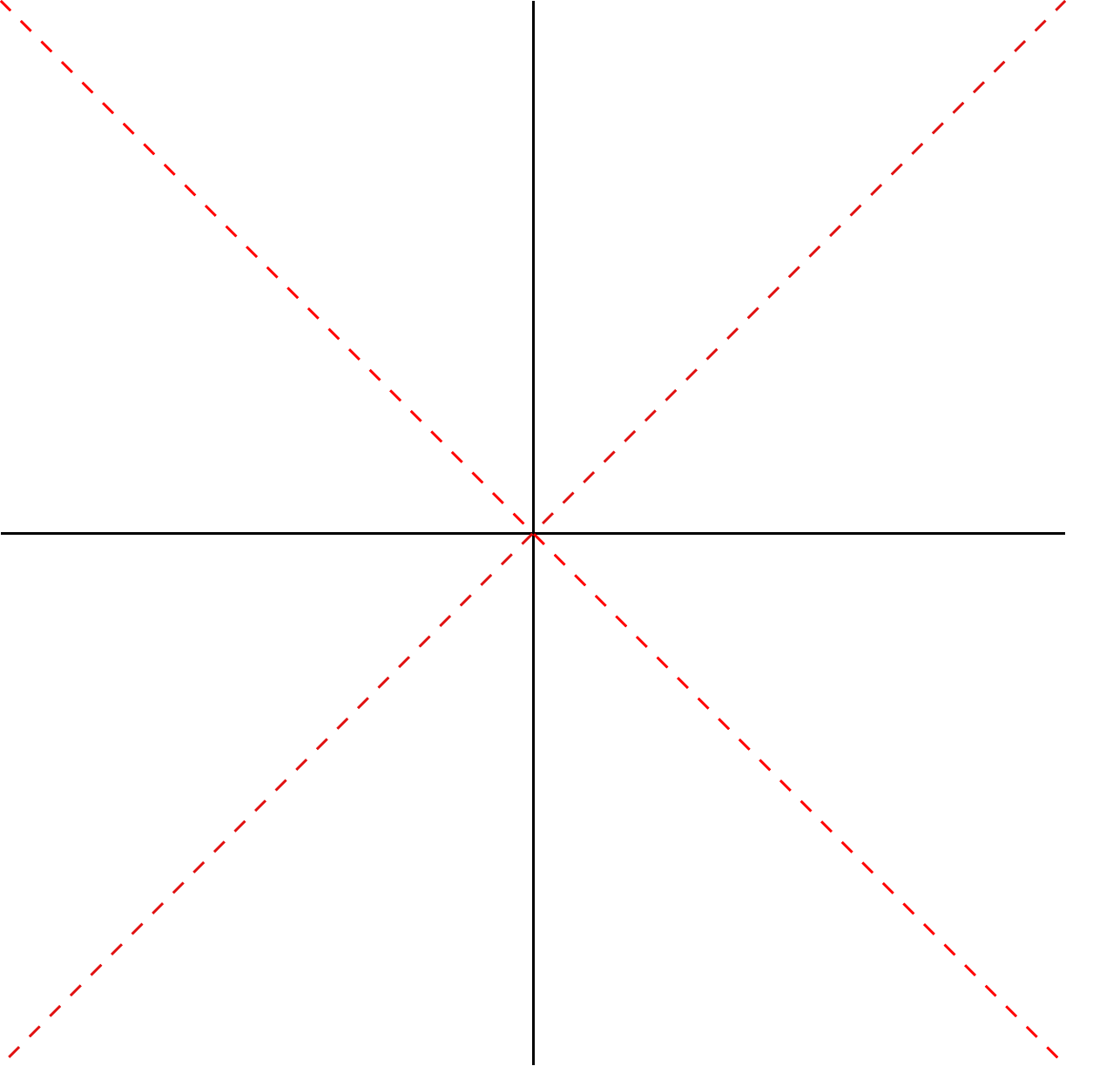}}%
    \put(0.85409757,0.45154967){\color[rgb]{0,0,0}\makebox(0,0)[lt]{\begin{minipage}{0.46664473\unitlength}\raggedright $w^1$\end{minipage}}}%
    \put(0.34856579,0.95060028){\color[rgb]{0,0,0}\makebox(0,0)[lt]{\begin{minipage}{0.25276589\unitlength}\raggedright $w^2$\end{minipage}}}%
  \end{picture}%
\endgroup%
 &
\def\svgwidth{3cm}\\
 &  \footnotesize{$\mathcal{S}: (w^1)^2-(w^2)^2=0$} &  \footnotesize{$\mathcal{S}: w^1=0$}\\[0.5ex]
\hline
\end{tabular}
\end{center}

Note that $(\C^{\e})^{2+n}-\mathcal{S}$ is connected and not
simply-connected for $\e=-1$ and $\lambda=0$ while it is not
connected nor simply-connected for the other values. Moreover,
$(\C^{\e})^{2+n}-\mathcal{S}$ has two connected components for
$\lambda=-\e/2$ and $\e=\pm 1$ and four connected components for
$\lambda=0$ and $\e=1$.

\bigskip

Finally we show that  degenerate homogeneous $\e$-K\"ahler
structures of linear type indeed exist and are realized in the
manifold $((\C^{\e})^{2+n}-\mathcal{S},g)$.
\begin{proposition}
For every data $(b,R_0)$ satisfying \eqref{differential equation b
general form }, the $\e$-K\"ahler manifold
$((\C^{\e})^{2+n}-\mathcal{S},g)$ admits a strongly degenerate
pseudo-K\"ahler homogeneous structure of linear type.
\end{proposition}

\begin{proof}
Let
$$\xi=\left\{\begin{array}{l}
\frac{-1}{(w^1)^2-\e(w^2)^2}(w^1\partial_{z^1}+w^2\partial_{z^2})\qquad
\lambda=0,\\
  \\
-\frac{1}{w^1}\partial_{z^1}\qquad \lambda=-\frac{\e}{2}.
\end{array}\right.$$
We take the tensor field
\begin{align*}
S_{X}Y &=  g(X,Y)\xi-g(\xi,Y)X+\e g(X,JY)J\xi-\e
g(\xi,JY)JX\nonumber\\ & - 2\lambda g(\xi,JX)JY.
\end{align*}
A straightforward computation shows that $\wnabla \xi=0$ and
$\wnabla R=0$, where $\wnabla=\nabla-S$, so that $S$ satisfies
equations \eqref{AS+J}.
\end{proof}

\section{The homogeneous model for a degenerate homogeneous $\e$-K\"ahler structure of linear
type}

Let $(M,g,J)$ be an $\e$-K\"ahler manifold admitting a degenerate
homogeneous structure of linear type $S$. From \cite{TV} one can
construct a Lie algebra of infitesimal isometries associated to
$S$. This algebra is (fixing a point $p\in M$ as the origin)
$$\f{g}=T_pM\oplus \f{hol}^{\widetilde{\nabla}}$$
where $\widetilde{\nabla}=\nabla-S$ is the canonical connection
associated to the homogeneous structure tensor $S$. The brackets
in $\f{g}$ are
$$\left\{\begin{array}{lcll}
\left[A,B\right] & = & AB-BA, & A,B\in \f{hol}^{\widetilde{\nabla}}\\
\left[A,\eta\right] & = & A\cdot\eta, &
A\in \f{hol}^{\widetilde{\nabla}},\eta\in T_pM\\
\left[\eta,\zeta\right] & = &
S_{\eta}\zeta-S_{\zeta}\eta-\widetilde{R}_{\eta\zeta}, &
\eta,\zeta\in T_pM,
\end{array}\right.$$
 where $\widetilde{R}$ is the
curvature tensor of $\widetilde{\nabla}$. This curvature tensor
can be computed as $\widetilde{R}=R-R^S$ where
$$R^S_{XY}Z=\left[S_X,S_Y\right]Z-S_{S_XY-S_YX}Z.$$
Let $G$ be a Lie group with Lie algebra $\f{g}$, and let $H$ be
the connected Lie subgroup with Lie algebra
$\f{hol}^{\widetilde{\nabla}}$. If $H$ is closed in $G$, then the
infinitesimal model $(\f{g},\f{hol}^{\widetilde{\nabla}})$ is
called \textit{regular}, and we shall call $G/H$ the
\textit{homogeneous model} for $M$ associated to $S$, which means
that $(M,g,J)$ is locally $\e$-holomorphically isometric to $G/H$
with the $G$-invariant metric and $\e$-complex structure given by
$g$ and $J$ at $T_pM$.

In \cite{CL} the infinitesimal model for $\e=-1$ and $\lambda=0$
is computed, proving that it is regular and the corresponding
homogeneous model is not complete. We shall obtain the same result
for the rest of the values of $\e$ and $\lambda$.

\subsection{The case $\lambda=-\frac{\e}{2}$}

Denoting $p_1=\xi$ and $p_2=J\xi$, for the sake of simplicity we
choose $p\in M$ be such that with respect to the basis
$\{p_1,p_2,q_1,q_2,X_a,JX_a\}$ and its dual
$\{p^1,p^2,q^1,q^2,X^a,JX^a\}$ the curvature is written
$$R_p=R_0q^1\wedge q^2\otimes(q^1\otimes p_2+\e q^2\otimes p_1).$$
Substituting $\lambda=-\e/2$ in \eqref{e-Kahler structure} we
obtain by direct calculation that the non-vanishing terms of
$\widetilde{R}$ are:
$$\begin{array}{rrcl}
\widetilde{R}_{p_2q_1}: & q_1 & \mapsto & 2p_2\\
                        & q_2 & \mapsto & 2\e p_1\\
                        & X_a & \mapsto & 0\\
                        & JX_a & \mapsto & 0\\
                        & p_1,p_2 & \mapsto & 0
\end{array}\hspace{1em}
\begin{array}{rrcl}
\widetilde{R}_{q_1q_2}: & q_1 & \mapsto & (R_0-b(p))p_2\\
                        & q_2 & \mapsto & (R_0-b(p))\e p_1\\
                        & X_a & \mapsto & -JX_a\\
                        & JX_a & \mapsto & -\e X_a\\
                        & p_1,p_2 & \mapsto & 0
\end{array}
$$
$$\begin{array}{rrcl}
\widetilde{R}_{q_2X_a}: & q_1 & \mapsto & -JX_a\\
                        & q_2 & \mapsto & -\e X_a\\
                        & X_a & \mapsto & p_2\\
                        & JX_a & \mapsto & \e p_1\\
                        & p_1,p_2 & \mapsto & 0
\end{array}\hspace{1em}
\begin{array}{rrcl}
\widetilde{R}_{q_2JX_a}:& q_1 & \mapsto & -\e X_a\\
                        & q_2 & \mapsto & -\e JX_a\\
                        & X_a & \mapsto & \e p_1\\
                        & JX_a & \mapsto & \e p_2\\
                        & p_1,p_2 & \mapsto & 0
\end{array}
$$
$$\begin{array}{rrcl}
\widetilde{R}_{X_aJX_a}: & q_1 & \mapsto & -2p_2\\
                        & q_2 & \mapsto & -2\e p_1\\
                        & X_a & \mapsto & 0\\
                        & JX_a & \mapsto & 0\\
                        & p_1,p_2 & \mapsto & 0,
\end{array}$$
so that $\mathrm{dim}(\f{hol}^{\widetilde{\nabla}})=2n+2$.
Choosing endomorphisms
$$A=2(q^1\otimes p_2+\e q^2\otimes p_1), \hspace{1em}B_a=\widetilde{R}_{q_2X_a},\hspace{1em} C_a=\widetilde{R}_{q_2JX_a},$$
$$K=\frac{1}{2}(R_0-b(p))A-\sum_a(X^a\otimes JX_a+\e JX^a\otimes X_a)$$
as basis of $\f{hol}^{\widetilde{\nabla}}$, the Lie algebra
$\f{g}$ has non-vanishing brackets
\begin{equation*}
\begin{array}{l}
\left[B_a,C_a\right]=\e
A,\hspace{1em}\left[B_a,K\right]=-C_a,\hspace{1em}\left[C_a,K\right]=-\e
B_a,\\
\left[A,q_1\right]=2p_2,\hspace{1em}\left[A,q_2\right]=2\e
p_1,\\
\left[B_a,q_1\right]=-JX_a,\hspace{1em}\left[B_a,q_2\right]=-\e
X_a,\hspace{1em}\left[B_a,X_a\right]=-p_2,\hspace{1em}\left[B_a,JX_a\right]=-\e
p_1,\\
\left[C_a,q_1\right]=-\e X_a,\hspace{1em}\left[C_a,q_2\right]=-\e
y_a,\hspace{1em}\left[C_a,X_a\right]=\e
p_1,\hspace{1em}\left[C_a,JX_a\right]=\e p_2,\\
\left[K,X_a\right]=JX_a,\hspace{1em}\left[K,JX_a\right]=\e X_a,\\
\left[p_1,q_1\right]=-p_1,\hspace{1em}\left[p_2,q_1\right]=-3p_2-A,\hspace{1em}\left[p_2,q_2\right]=-2\e
p_1,\\
\left[q_1,q_2\right]=2b(p)p_2-q_2-\frac{1}{2}(R_0-b(p))A+K,\\
\left[q_1,X_a\right]=X_a,\hspace{1em}\left[q_1,JX_a\right]=JX_a,\\
\left[q_2,X_a\right]=2JX_a-B_a,\hspace{1em}\left[q_2,JX_a\right]2\e
X_a-C_a,\\
\left[X_a,JX_a\right]=2p_2+A.
\end{array}\end{equation*}
One can check that $\f{g}$ is a solvable Lie algebra with a
$3$-step nilradical
$\f{n}=\mathrm{span}\left\{p_1,p_2,q_2-K,X_a,JX_a,A,B_a,C_a,a=1,\ldots,n\right\}$.
Since $\f{g}$ has trivial center, the adjoint representation is
faithful and provides a matrix realization of $\f{g}$. With
respect to this realization it is a straightforward computation
that by exponentiation of $\f{g}$ and $\f{hol}^{\wnabla}$ we
obtain a Lie group $G$ and a closed Lie subgroup $H$ respectively,
so that the infinitesimal model $(\f{g},\f{hol}^{\wnabla})$ is
regular and $G/H$ is a homogeneous model for $(M,g,J)$. Let
$\bar{g}$ and $\bar{J}$ be the $G$-invariant metric and complex
structure on $G/H$ induced from $(M,g,J)$ on $G/H$.

\begin{proposition}\label{proposition completeness}
The homogeneous model $(G/H,\bar{g},\bar{J})$ is not geodesically
complete.
\end{proposition}

\begin{proof}
Let $\sigma$ be the Lie algebra involution of $\f{g}$ given by
$$\begin{array}{rrcl}
\sigma: & \f{g} & \to & \f{g}\\
        & A & \mapsto & -A\\
        & B_a & \mapsto & -B_a\\
        & C_a & \mapsto & C_a\\
        & K   & \mapsto & -K\\
        & p_1 & \mapsto & p_1\\
        & p_2 & \mapsto & -p_2\\
        & q_1 & \mapsto & q_1\\
        & q_2 & \mapsto & -q_2\\
        & X_a & \mapsto & X_a\\
        & JX_a & \mapsto & -JX_a.
\end{array}$$
One can check that the restriction of $\sigma$ to $\f{m}$ is an
isometry with respect to the bilinear form given by $\bar{g}$. The
subalgebra of fixed points is
$$\f{g}^{\sigma}=\mathrm{span}\left\{p_1,q_1,X_a,C_a,a=1,\ldots,n\right\}.$$
Working with the universal cover if necessary we can assume that
$G$ is simply-connected so that $\sigma$ induces an involution in
$G$ and therefore an isometric involution in $G/H$. We will denote
all this involutions by $\sigma$. Let $G^{\sigma}$ be the
connected Lie subgroup of $G$ with Lie algebra $\f{g}^{\sigma}$,
note that
$\sigma(\f{hol}^{\widetilde{\nabla}})\subset\f{hol}^{\wnabla}$, so
that $\left(G/H\right)^{\sigma}=G^{\sigma}/H^{\sigma}$, where the
superindex $\sigma$ stands for the fixed point set by $\sigma$. It
is a well-known result that $\left(G/H\right)^{\sigma}$ is a
closed totally geodesic submanifold of $G/H$. Let now $\theta$ be
the Lie algebra involution of $\f{g}^{\sigma}$ given by
$$\begin{array}{rrcl}
\theta: & \f{g}^{\sigma} & \to & \f{g}^{\sigma}\\
        & C_a & \mapsto & -C_a\\
        & p_1 & \mapsto & p_1\\
        & q_1 & \mapsto & q_1\\
        & X_a & \mapsto & -X_a,
\end{array}$$
which is again an isometry with respect to the bilinear form
induced in $\f{g}^{\sigma}$ by restriction from $\f{m}$. The
subalgebra of fixed points is
$\f{k}=\left(\f{g}^{\sigma}\right)^{\theta}=\mathrm{span}\left\{p_1,q_1\right\}$.
Note that $\f{k}\cap\f{hol}^{\wnabla}=0$. Let
$\widetilde{G^{\sigma}}$ be the universal cover of $G^{\sigma}$
and $\widetilde{H^{\sigma}}$ the corresponding closed subgroup,
$\theta:\f{g}^{\sigma} \to \f{g}^{\sigma}$ induces an isometric
involution
$\theta:\widetilde{G^{\sigma}}/\widetilde{H^{\sigma}}\to\widetilde{G^{\sigma}}/\widetilde{H^{\sigma}}$.
Therefore, let $K$ be the connected Lie subgroup of
$\widetilde{G^{\sigma}}$ with lie algebra $\f{k}$, $K$ is a
totally geodesic submanifold of
$\widetilde{G^{\sigma}}/\widetilde{H^{\sigma}}$. Let $s$ be the
sign of $b(p)$. We define the left-invariant vector fields
$U=1/(\sqrt{|b(p)|})q_1$, $V=U-s\sqrt{|b(p)|}p_1$ in $\f{k}$. We
have $<U,U>=s$, $<V,V>=-s$, $<U,V>=0$, and
$[U,V]=\frac{1}{\sqrt{|b(p)|}}(V-U)$, where
$<\hspace{1mm},\hspace{1mm}>$ stands for the bilinear form
inherited by $\f{k}$ from $\f{g}^{\sigma}$.
%The Levi-Civita
%connection of this metric is
%$$\begin{array}{cc}
%\nabla_UU=-\frac{1}{\sqrt{|b(p)|}}V, &
%\nabla_UV=-\frac{1}{\sqrt{|b(p)|}}U,\\
%\nabla_VV=-\frac{1}{\sqrt{|b(p)|}}U, &
%\nabla_VU=-\frac{1}{\sqrt{|b(p)|}}V.
%\end{array}$$
%Let $\gamma$ be a curve in $K$ and $\dot{\gamma}$ its tangent
%vector. Setting $\dot{\gamma}(t)=u(t)U+v(t)V$, the geodesic
%equation $\nabla_{\dot{\gamma}}\dot{\gamma}=0$ implies
%$$
%\begin{array}{l}
%\dot{u}-\frac{1}{\sqrt{|b(p)|}}(uv+v^2)=0\\
%\dot{v}-\frac{1}{\sqrt{|b(p)|}}(uv+u^2)=0.
%\end{array}
%$$
%Changing variables to $x=u+v$ and $y=v-u$ the equations transform
%into
%$$
%\begin{array}{l}
%\dot{x}-\frac{1}{\sqrt{|b(p)|}}x^2=0\\
%\dot{y}+\frac{1}{\sqrt{|b(p)|}}xy=0.
%\end{array}
%$$
%The solutions for $x$ is
%$$x=\sqrt{|b(p)|}\frac{1}{c-t},$$
%for some constant $c\in\bb{R}$. Therefore, $K$ is not geodesically
%complete.
It is a straightforward computation to see that $K$ is not
geodesically complete. Hence, since we have the following
inclusions of totally geodesic submanifolds
$$K\subset\widetilde{G^{\sigma}},\qquad G^{\sigma}=\left(G/H\right)^{\sigma}\subset G/H,$$
the manifold $(G/H,g,J)$ is not geodesically complete.
\end{proof}

\begin{corollary}
Let $(M,g,J)$ be a connected and simply-connected $\e$-K\"ahler
manifold admitting a degenerate $\e$-K\"ahler structure of linear
type with $\zeta=-\frac{\e}{2}\xi$, then it is not geodesically
complete.
\end{corollary}

\begin{proof}
Suppose that $(M,g,J)$ is geodesically complete, then
Ambrose-Singer theorem assures that $(M,g,J)$ is (globally)
$\e$-holomorphically isometric to the homogeneous model
$(G/H,\bar{g},\bar{J})$. But this homogeneous model is not
geodesically complete.
\end{proof}

\subsection{The case $\lambda=0$}

Denoting again $p_1=\xi$ and $p_2=J\xi$, for the sake of
simplicity we choose $p\in M$ be such that with respect to the
basis $\{p_1,p_2,q_1,q_2,X_a,JX_a\}$ and its dual
$\{p^1,p^2,q^1,q^2,X^a,JX^a\}$ the curvature is written
$$R_p=R_0q^1\wedge q^2\otimes(q^1\otimes p_2+\e q^2\otimes p_1).$$
Substituting $\lambda=0$ in \eqref{e-Kahler structure} we obtain
by direct calculation that the non-vanishing terms of
$\widetilde{R}$ are:
$$\begin{array}{rrcl}
\widetilde{R}_{p_1q_2}: & q_1 & \mapsto & -2p_2\\
                        & q_2 & \mapsto & -2\e p_1\\
                        & X_a & \mapsto & 0\\
                        & JX_a & \mapsto & 0\\
                        & p_1,p_2 & \mapsto & 0
\end{array}\hspace{1em}
\begin{array}{rrcl}
\widetilde{R}_{p_2q_1}: & q_1 & \mapsto & 2p_2\\
                        & q_2 & \mapsto & 2\e p_1\\
                        & X_a & \mapsto & 0\\
                        & JX_a & \mapsto & 0\\
                        & p_1,p_2 & \mapsto & 0
\end{array}$$
$$\begin{array}{rrcl}
\widetilde{R}_{q_1q_2}: & q_1 & \mapsto & (R_0-2b(p))p_2\\
                        & q_2 & \mapsto & (R_0-2b(p))\e p_1\\
                        & X_a & \mapsto & 0\\
                        & JX_a & \mapsto & 0\\
                        & p_1,p_2 & \mapsto & 0
\end{array}\hspace{1em}
\begin{array}{rrcl}
\widetilde{R}_{X_aJX_a}: & q_1 & \mapsto & -2p_2\\
                        & q_2 & \mapsto & -2\e p_1\\
                        & X_a & \mapsto & 0\\
                        & JX_a & \mapsto & 0\\
                        & p_1,p_2 & \mapsto & 0,
\end{array}$$
so that $\mathrm{dim}(\f{hol}^{\widetilde{\nabla}})=1$. Choosing
the  endomorphism
$$A=2(q^1\otimes p_2+\e q^2\otimes p_1)$$
as basis of $\f{hol}^{\widetilde{\nabla}}$, the Lie algebra
$\f{g}$ has non-vanishing brackets
\begin{equation*}
\begin{array}{l}
\hspace{1em}\left[A,q_1\right]=2p_2,\hspace{1em}\left[A,q_2\right]=2\e
p_1,\\
\left[p_1,q_1\right]=-p_1,\hspace{1em}\left[p_1,q_2\right]=p_2+A,\\
\left[p_2,q_1\right]=-3p_2-A,\hspace{1em}\left[p_2,q_2\right]=-\e
p_1,\\
\left[q_1,q_2\right]=2b(p)p_2-\frac{1}{2}(R_0-2b(p))A,\\
\left[q_1,X_a\right]=X_a,\hspace{1em}\left[q_1,JX_a\right]=JX_a,\\
\left[q_2,X_a\right]=JX_a,\hspace{1em}\left[q_2,JX_a\right]=\e
X_a,\\
\left[X_a,JX_a\right]=2p_2+A.
\end{array}\end{equation*}
One can check that $\f{g}$ is a solvable Lie algebra with a
$2$-step nilradical
$\f{n}=\mathrm{span}\left\{p_1,p_2,X_a,JX_a,A,a=1,\ldots,n\right\}$.
Since $\f{g}$ has trivial center, the adjoint representation is
faithful and provides a matrix realization of $\f{g}$.
Exponentiating $\f{g}$ and $\f{hol}^{\wnabla}$ we obtain a Lie
group $G$ and a closed Lie subgroup $H$ respectively, so that the
infinitesimal model $(\f{g},\f{hol}^{\wnabla})$ is regular and
$G/H$ is a homogeneous model for $(M,g,J)$. Let $\bar{g}$ and
$\bar{J}$ be the $G$-invariant metric and complex structure on
$G/H$ induced from $(M,g,J)$ on $G/H$.

\begin{proposition}
The homogeneous model $(G/H,\bar{g},\bar{J})$ is not geodesically
complete.
\end{proposition}

\begin{proof}
Following the same arguments as in the proof of Proposition
\ref{proposition 2} (see also those in \cite[\S 4]{CL}) we find
isometric involutions $\sigma:\f{g}\to\f{g}$ and
$\f{g}^{\sigma}\to\f{g}^{\sigma}$ so that the connected Lie
subgroup with lie algebra
$\f{k}=\left(\f{g}^{\sigma}\right)^{\theta}$ is not complete.
\end{proof}

\begin{corollary}
Let $(M,g,J)$ be a connected and simply-connected $\e$-K\"ahler
manifold admitting a strongly-degenerate $\e$-K\"ahler structure
of linear type, then it is not geodesically complete.
\end{corollary}

\section{The $\e$-quaternion K\"ahler case}

Throughout this section $\mathrm{dim}(M)=4n\geq 8$ is assumed. We
shall study degenerate homogeneous structures of linear type in
the $\e$-quaternion K\"ahler case.

Let $\e=(\e_1,\e_2,\e_3)$, we can combine some definitions of
pseudo-quaternion geometry and para-quaternion geometry in the
following way. For pseudo-quaternion geometry $\e$ must be
substituted by $(-1,-1,-1)$, and for para-quaternion geometry $\e$
must be substituted by $(-1,1,1)$.
\begin{definition}
Let $(M,g)$ be a pseudo-Riemannian manifold.
\begin{enumerate}
\item  An $\e$-quaternion Hermitian structure on $(M,g)$ is a
$3$-rank subbundle $Q\subset \f{so}(TM)$ with a local basis
$J_1,J_2,J_3$ satisfying
$$J_a^2=\e_a,\qquad J_1J_2=J_3.$$

\item \( (M,g) \) is called \( \e \)-quaternion K\"ahler if it is
strongly-oriented and it admits a parallel \( \e \)-quaternion
Hermitian structure with respect to the Levi-Civita connection.
\end{enumerate}
\end{definition}
The first previous definition means that at every point $p\in M$
there is a subalgebra $Q_p\subset\f{so}(T_pM)$ isomorphic to the
imaginary $\e$-quaternions, and in particular $g$ has signature
$(4r,4s)$, $r+s=n$ for $\e=(-1,-1,-1)$ and $(2n,2n)$ for
$\e=(-1,1,1)$. We shall denote by \( Sp^{\e}(n) \) the group \(
\Sp(r,s) \), \( r+s=n \), when \( \e=(-1,-1,-1) \) and \(
\Sp(n,\R) \) when \( \e=(-1,1,1) \). Their Lie algebras are
denoted by \( \sP^{\e}(n) \) respectively. For the proof of the
following Proposition see \cite{AC}.

\begin{proposition}\label{proposition quaternion kahler curvature}
\emph{\cite{AC}} An $\e$-quaternion K\"ahler manifold is Einstein
and has Riemann curvature tensor
$$R=\nu_q R_0+R^{\f{sp}^{\e}(n)},$$
where $\nu_q=\frac{\mathrm{s}}{16n(n+2)}$ is one quarter the
reduced scalar curvature, $R_0$ is four times the curvature of the
$\e$-quaternionic hyperbolic space (of the corresponding
signature)
\begin{align}\label{R_0}
(R_0)_{XYZW} & =
g(X,Z)g(Y,W)-g(Y,Z)g(X,W)-\sum_{a}\e_a\left\{g(J_aX,Z)g(J_aY,W)\right.\nonumber\\
 &  \left.-g(J_aY,Z)g(J_aX,W)+2g(X,J_aY)g(Z,J_aW)\right\},
\end{align}
and $R^{\f{sp}^{\e}(n)}$ is an algebraic curvature tensor of type
$\f{sp}^{\e}(n)$.
\end{proposition}

Let $J_1,J_2,J_3$ be a local basis of $Q$, and
$\omega_a=g(\cdot,J_a\cdot)$, $a=1,2,3$. The $4$-form
$$\Omega=\sum_a-\e_a\omega_a\wedge\omega_a$$
is independent of the choice of basis and hence it is globally
defined. An  $\e$-quaternion Hermitian manifold $(M,g,Q$ is
$\e$-quaternion K\"ahler if and only if $\Omega$ is parallel with
respect to the Levi-Civita connection (cf. \cite{AC}), or
equivalently if the holonomy of the Levi-Civita connection is
contained in $Sp^{\e}(n)Sp^{\e}(1)$. This is also equivalent to
\begin{equation}\label{nabla Ja}\nabla
J_a=\sum_{b=1}^3c_{ab}J_b,\qquad a=1,2,3,\end{equation} where
$(c_{ab})$ is a matrix in $\f{sp}^{\e}(1)$.

\begin{definition} An $\e$-quaternion K\"ahler
manifold $(M,g,Q)$ is called a homogeneous $\e$-quaternion
K\"ahler manifold if there is a connected Lie group $G$ of
isometries acting transitively on $M$ and preserving $Q$.
$(M,g,Q)$ is called a reductive homogeneous $\e$-quaternion
K\"ahler manifold if the Lie algebra $\f{g}$ of $G$ can be
decomposed as $\f{g}=\f{h}\oplus\f{m}$ with
$$[\f{h},\f{h}]\subset \f{h},\qquad[\f{h},\f{m}]\subset \f{m}.$$
\end{definition}
As a corollary of Kiri\v{c}enko's Theorem \cite{Kir} we have
\begin{theorem}
A connected, simply-connected and (geodesically) complete
$\e$-quaternion K\"ahler manifold $(M,g,Q)$ is reductive
homogeneous if and only if it admits a linear connection $\wnabla$
satisfying \begin{equation}\label{AS qK}\wnabla g=0,\qquad \wnabla
R=0,\qquad \wnabla S=0,\qquad \wnabla\Omega=0,\end{equation} where
$S=\nabla-\wnabla$, $\nabla$ is the Levi-Civita connection, $R$ is
the curvature tensor of $\nabla$, and $\Omega$ is the canonical
$4$-form associated to $Q$.
\end{theorem}
Note that the condition $\wnabla\Omega=0$ is equivalent to
\begin{equation}\label{wnabla Ja}
\wnabla J_a=\sum_{b=1}^3\widetilde{c}_{ab}J_b,\qquad a=1,2,3,
\end{equation}
where $(\widetilde{c}_{ab})\in\f{sp}^{\e}(1)$. A tensor field $S$
satisfying the previous equations is called a \emph{homogeneous
$\e$-quaternion K\"ahler structure}. The classification of such
structures was obtained in \cite{BGO} and \cite{CL2} , resulting
five primitive classes $\QK^{\e}_1,\QK^{\e}_2$,
$\QK^{\e}_3,\QK^{\e}_4,\QK^{\e}_5$. Among them
$\QK^{\e}_1,\QK^{\e}_2,\QK^{\e}_3$ have dimension growing linearly
with respect to the dimension of $M$. Hence
%(see the same reference for the notation). The corresponding
%classes are
%\begin{align*}
%\QK_1 &  = \Bigl\{ S \in \cV \,:\,S_{XYZ}= \sum_{a=1}^3\,
%\theta(J_aX) \left \langle J_aY,Z\right\rangle,\;\theta\in
%V^{\ast} \Bigr\},\\
%\QK_2 & =\Bigl\{\, S\in\cV \,:\,S_{XYZ}= \sum_{a=1}^3 \theta^a(X)
%\left \langle J_aY,Z \right \rangle,\; \sum_{a=1}^3 \theta^a \circ
%J_a=0,\\
%& \hspace{79mm} \;\theta^a \in V^{\ast} \,\Bigr\},\\
%\QK_3 &  = \Bigl\{ S \in \hV \,:\, S_{XYZ}=\langle
%X,Y \rangle\theta(Z)-\langle X,Z \rangle \theta(Y)\\
%&  \hspace{20mm} + \sum_{a=1}^3 \bigl(\langle X,J_aY\rangle
%\theta(J_aZ)
%                      - \langle X,J_aZ \rangle \theta(J_aY)\bigr),\,\theta\in V^{\ast} \Bigl\},\\
%\QK_4 &  = \Bigl\{ S \in \hV : S_{XYZ}
%= \frac12 \bigl( S_{YZX} + S_{ZXY} \\
%&  \hspace{31mm} + \sum_{a=1}^3 (S_{J_aYJ_aZX} + S_{J_aZXJ_aY})\bigr), \,c_{12}(S)=0 \Bigl\}, \\
%\QK_5 &  = \Bigl\{ S \in \hV \,:\, S_{XYZ} = -\frac14 \bigl( S_{YZX} + S_{ZXY} \\
%&  \hspace{49mm} + \sum_{a=1}^3 (S_{J_aYJ_aZX} + S_{J_aZXJ_aY})
%\bigr) \Bigl\}.
%\end{align*}
%with $\mathrm{dim}(\QK_1)=\mathrm{dim}(\QK_3)=4n$ and
%$\mathrm{dim}(\QK_2)=8n$, where $\mathrm{dim}(M)=4n$. For this
%reason we give the following
\begin{definition}
A homogeneous $\e$-quaternion K\"ahler structure $S$ is called of
linear type if it belongs to the class
$\QK^{\e}_1+\QK^{\e}_2+\QK^{\e}_3$.
\end{definition}
The local expression of $S\in\QK^{\e}_1+\QK^{\e}_2+\QK^{\e}_3$ is
\begin{multline}\label{struct linear type
e-QT}S_XY=g(X,Y)\xi-g(Y,\xi)X-\sum_{a=1}^3\e_a\left(g(J_aY,\xi)J_aX-g(X,J_aY)J_a\xi\right)\\
+\sum_{a=1}^3g(X,\zeta^a)J_aY,
\end{multline}
where $\xi$ and $\zeta^a$, $a=1,2,3$, are vector fields. We then
give the following further definition.
\begin{definition}
A homogeneous $\e$-quaternion K\"ahler structure of linear type is
called degenerate if $\xi\neq 0$ and $g(\xi,\xi)=0$.
\end{definition}
\begin{remark}
The case $\zeta^a=0$ for $a=1,2,3$ was studied in \cite{CL}
resulting that the manifold $(M,g,Q)$ must be flat.
\end{remark}
\begin{proposition}
Let $(M,g,Q)$ be a $\e$-quaternion K\"ahler manifold admitting a
degenerate homogeneous $\e$-quaternion K\"ahler structure of
linear type. Then $(M,g,Q)$ is flat.
\end{proposition}

\begin{proof}
Following Proposition \ref{proposition quaternion kahler
curvature} we decompose the curvature tensor field of $(M,g,Q)$ as
$R=\nu_qR_0+R^{\f{sp}^{\e}(n)}$. Recall that the space of
algebraic curvature tensors $\mathcal{R}^{\f{sp}^{\e}(n)}$ is
$[S^4E]$ with $E=\C^{2n}$ for $\e=(-1,-1,-1)$, and $S^4E$ with
$E=\R^{2n}$ for $\epsilon=(-1,1,1)$. Since $R_0$ is
$Sp^{\epsilon}(n)Sp^{\epsilon}(1)$-invariant and $\nu_q$ is
constant, the covariant derivative $\nabla R_0$ vanishes.
Moreover, for every vector field $X$, $S_X$ acts as an element of
$\f{sp}^{\e}(n)+\f{sp}^{\e}(1)$, whence $S R_0=0$.  Using the
second equation of \eqref{AS qK} and $\wnabla=\nabla-S$ we have
that
$$0=\wnabla R=\nu_q\wnabla R_0+\wnabla R^{\f{sp}^{\e}(n)}=\nabla R^{\f{sp}^{\e}(n)}-S R^{\f{sp}^{\e}(n)}.$$
Writing
$T^*M\otimes(\f{sp}^{\e}(n)+\f{sp}^{\e}(1))=T^*M\otimes\f{sp}^{\e}(n)+T^*M\otimes\f{sp}^{\e}(1)$
we can decompose $S=S_E+S_H$, and hence $S_H
R^{\f{sp}^{\e}(n)}=0$. We thus obtain
\begin{equation*}
\nabla R=\nabla R^{\f{sp}^{\e}(n)}=S_E R^{\f{sp}^{\e}(n)},
\end{equation*}
which we can write as
\begin{equation}\label{formula 2 AS}(\nabla_XR)_{YZWU}=-R^{\f{sp}^{\e}(n)}_{S_XYZWU}-R^{\f{sp}^{\e}(n)}_{YS_XZWU}-R^{\f{sp}^{\e}(n)}_{YZS_XWU}-R^{\f{sp}^{\e}(n)}_{YZWS_XU}.\end{equation}
Taking the cyclic sum in $X,Y,Z$ and applying Bianchi identities
we obtain
\begin{align*}
0 & =
\Cyclic{\fdS}{\SXYZ}\left\{2g(X,\xi)R^{\f{sp}^{\e}(n)}_{YZWU}+g(X,W)R^{\f{sp}^{\e}(n)}_{YZ\xi
U}+g(X,U)R^{\f{sp}^{\e}(n)}_{YZW\xi}\right.\\
 &  \left.+2\sum_a\e_a( g(X,J_aY)R^{\f{sp}^{\e}(n)}_{J_a\xi
 ZWU}+
g(X,J_aW)R^{\f{sp}^{\e}(n)}_{YZJ_a\xi U}+
g(X,J_aU)R^{\f{sp}^{\e}(n)}_{YZWJ_a\xi})\right\}.
\end{align*}
Contracting the previous formula with respect to $X$ and $W$, and
taking into account that $R^{\f{sp}^{\e}(n)}$ is traceless we
obtain
$$(4n+2)R^{\f{sp}^{\e}(n)}_{YZ\xi U}=0,$$
for every vector fields $Z,Y,U$. Expanding the expression of $S$
in \eqref{formula 2 AS} and using the previous formula we arrive
at
$$0=\Cyclic{\fdS}{\SXYZ}\theta(X)R^{\f{sp}^{\e}(n)}_{YZWU},$$
where $\theta=\xi^{\flat}$, or equivalently
\begin{equation}\label{wedge 1}0=\theta\wedge R^{\f{sp}^{\e}(n)}_{WU}.\end{equation}
Noting that $R^{\f{sp}^{\e}(n)}$ satisfies the symmetries
$R^{\f{sp}^{\e}(n)}_{XJ_aYWU}+R^{\f{sp}^{\e}(n)}_{J_aXYWU}=0$,
$a=1,2,3$, we will also have \begin{equation}{\label{wedge
2}}0=(\theta\circ J_a)\wedge R^{\f{sp}^{\e}(n)}_{WU}=0,\qquad
a=1,2,3.\end{equation} It is easy to prove that a curvature tensor
of type $\f{sp}^{\e}(n)$ satisfying equations \eqref{wedge 1} and
\eqref{wedge 2} must vanish. Therefore we conclude that
$R=\nu_qR_0$.

Now, using the third equation of \eqref{AS qK} together with
\eqref{struct linear type e-QT}, and taking into account
\eqref{wnabla Ja} we have
\begin{align*}
0 &=
g(X,Y)\wnabla_Z\xi-g(\wnabla_Z\xi,Y)X-\sum_a\epsilon_a(g(\wnabla_Z\xi,J_aY)J_aX+g(X,J_aY)J_a\wnabla_Z\xi)\\
&
+\sum_ag(X,\wnabla_Z\zeta^a-\sum_b\widetilde{c}_{ba}(Z)\zeta^b)J_aY,
\end{align*}
whence we deduce that $\wnabla \xi=0$. On the other hand
\begin{align*}
R_{XY}\xi & =
-\nabla_X\nabla_Y\xi+\nabla_Y\nabla_X\xi+\nabla_{[X,Y]}\xi\nonumber\\
& =
-g(Y,\nabla_X\xi)\xi-g(Y,\xi)\nabla_X\xi+g(X,\nabla_Y\xi)\xi+g(X,\xi)\nabla_Y\xi\nonumber\\
& -\sum_a\e_a\left(
g(Y,\nabla_XJ_a\xi)J_a\xi+g(Y,J_a\xi)\nabla_XJ_a\xi\right.
\nonumber\\
& \left.-g(X,\nabla_YJ_a\xi)J_a\xi-g(X,J_a\xi)\nabla_YJ_a\xi\right) \nonumber\\
&+\sum_a-g(Y,\nabla_X\zeta^a)J_a\xi-g(Y,\zeta^a)\nabla_XJ_a\xi
\nonumber\\
& +g(X,\nabla_Y\zeta^a)J_a\xi+g(X,\zeta^a)\nabla_YJ_a\xi.
\end{align*}
Therefore, applying $\wnabla\xi=0$ and \eqref{nabla Ja} to this
formula we see that
$R_{XY}\xi\in\mathrm{span}\{\xi,J_1\xi,J_2\xi,J_3\xi\}$. Comparing
this fact with $R_{XY}\xi=\nu_q(R_0)_{XY}\xi$ we obtain that
$\nu_q=0$, so that $R=0$.
\end{proof}

\section{Homogeneous structures of linear type}

The aim of this section is to bring together all results for
homogeneous structures of linear type in the purely
pseudo-Riemannian, $\e$-K\"ahler, and $\e$-quaternion K\"ahler
cases, in order to provide a general picture and a complete study
of this kind of structures. It is worth noting how different the
non-degenerate and the degenerate cases are, being the former
closely related with the well known case of definite metrics and
spaces of constant curvature, while the latter is rather related
to the geometry of singular homogeneous plane waves.

\bigskip

\subsection{The general picture}

For the purely pseudo-Riemannian case, the following theorem was
subsequently proved in \cite{TV,GO,Mon}. We recall that in this
setting a homogeneous pseudo-Riemannian structure is of linear
type if it is of the form $S_XY=g(X,Y)\xi-g(\xi,Y)X$, for some
vector field $\xi$. $S$ is degenerate if $g(\xi,\xi)=0$ and
non-degenerate otherwise.

\begin{theorem}\label{theorem linear pseudo-Riemannian}
Let $(M,g)$ be a pseudo-Riemannian manifold of dimension $m\geq
3$.
\begin{enumerate}
\item If $(M,g)$ admits a non-degenerate homogeneous structure of
linear type, then it has constant sectional curvature
$c=-g(\xi,\xi)$. Conversely, every non-flat simply-connected
pseudo-Riemannian space form locally admits a non-degenerate
homogenous structure of linear type, being this structure globally
defined if and only if $g$ is definite.

\item If $(M,g)$ admits a degenerate homogeneous structure of
linear type, then $(M,g)$ is locally isometric to a singular
scale-invariant homogeneous plane wave. Conversely, every singular
scale-invariant homogeneous plane wave admits a degenerate
homogeneous structure of linear type.
\end{enumerate}
\end{theorem}

For the $\e$-K\"ahler and $\e$-quaternion K\"ahler cases, the
non-degenerate case was obtained in \cite{LS} (see \cite{CGS0,GMM}
for definite metrics), and the degenerate case was partially solve
in \cite{CL} and completed in the present manuscript.

\begin{theorem}\label{theorem linear e-Kahler}
Let $(M,g,J)$ be an $\e$-K\"ahler manifold of dimension $m\geq 4$.
\begin{enumerate}
\item If $(M,g,J)$ admits a non-degenerate homogeneous
$\e$-K\"ahler structure of linear type, then it has constant
$\e$-holomorphic sectional curvature $c=-4g(\xi,\xi)$ and $\zeta
=0$. Conversely, every non-flat simply-connected $\e$-complex
space form locally admits a non-degenerate homogenous
$\e$-K\"ahler structure of linear type, being this structure
globally defined if and only if $g$ is definite.

\item If $(M,g,J)$ admits a degenerate homogeneous $\e$-K\"ahler
structure of linear type, then $\zeta = \lambda \xi$, with
$\lambda \in \{0,-\epsilon /2\}$, and $(M,g)$ is locally
$\e$-holomorphically isometric to $(\C^{\e})^{n+2}$ with the
metric
\begin{align}\label{local metric general theorem}
\bar{g}  &= dw^1dz^1-\e dw^2dz^2+b(dw^1dw^1-\e dw^2dw^2)\nonumber\\
& \hspace{50mm}+\sum_{a=1}^n\varepsilon^a(dx^adx^a-\e dy^ady^a),
\end{align}
where $\varepsilon^a=\pm 1$, and the function $b$ only depends on
the coordinates $\{w^1,w^2\}$ and satisfies
\begin{equation*}\Delta^{\e}b=\frac{R_0}{\|w\|_\lambda ^4} \end{equation*} for
$R_0\in\R-\{0\}$ and $\|w\|_\lambda$ defined in \eqref{norma}.
Conversely, $((\C^{\e})^{n+2},\bar{g})$ admits a degenerate
homogeneous $\e$-K\"ahler structure of linear type.
\end{enumerate}
\end{theorem}

\begin{theorem}
Let $(M,g,Q)$ be an $\e$-quaternion K\"ahler manifold of dimension
$m\geq 8$.
\begin{enumerate}
\item If $(M,g,Q)$ admits a non-degenerate homogeneous
$\e$-quaternion K\"ahler structure of linear type, then it has
constant $\e$-quaternionic sectional curvature $c=-4g(\xi,\xi)$
and $\zeta ^a =0$. Conversely, every non-flat simply-connected
$\e$-quaternion space form locally admits a non-degenerate
homogenous $\e$-quaternion K\"ahler structure of linear type,
being this structure globally defined if and only if $g$ is
definite.

\item If $(M,g,Q)$ admits a degenerate homogeneous $\e$-quaternion
K\"ahler structure of linear type, then $(M,g,Q)$ is flat.
\end{enumerate}
\end{theorem}

\subsection{Relation with homogeneous plane waves}

We exhibit the parallelism between certain kind of (Lorentzian)
homogeneous plane waves and Lorentz-K\"ahler spaces admitting
degenerate homogeneous structures of linear type (by
Lorentz-K\"ahler we mean pseudo-K\"ahler of index $2$). Although,
as far as the authors know, there is no formal definition of a
 plane wave in complex geometry, this relation could allow us to understand
the latter spaces as a complex generalization of the former, at
least in the important Lorentz-K\"ahler case, suggesting a
starting point for a possible definition of \textit{complex plane
wave}.

A plane wave is a Lorentz manifold $M=\R^{n+2}$ with metric
$$g=dudv+A_{ab}(u)x^ax^bdu^2+\sum_{a=1}^n(dx^a)^2,$$
where $(A_{ab})(u)$ is a symmetric matrix called the profile of
$g$. Moreover, a plane wave is called \textit{homogeneous} if the
Lie algebra of Killing vector fields acts transitively in the
tangent space at every point. Among homogeneous plane waves we
will be interested in two types. A \textit{Cahen-Wallach space} is
defined as a plane wave with profile a constant symmetric matrix
$(B_{ab})$, which makes it symmetric and geodesically complete. On
the other hand, a \textit{singular scale-invariant homogeneous
plane wave} is a plane wave with profile $(B_{ab})/u^2$, where
$(B_{ab})$ is a constant symmetric matrix. Singular
scale-invariant homogeneous plane waves are homogeneous but not
symmetric. In addition they present a singularity in the
cosmological sense at $\{u=0\}$, since the geodesic deviation
equation (or Jacobi equation) governed by the curvature is
singular at this set (see \cite{Sen}). Note that these two kind of
spaces are composed by the twisted product of a totally geodesic
flat wave front and a 2-dimensional manifold containing time and
the direction of propagation. This 2-dimensional space gives the
real geometric information of the total manifold and in particular
it contains a null parallel vector field. They are all VSI and the
curvature information is contained in the profile $A_{ab}(u)$,
since the only non-vanishing component of the curvature is given
by $R_{uaub}=-A_{ab}(u)$. By Theorem \ref{theorem linear
pseudo-Riemannian} singular scale-invariant homogeneous plane
waves are characterized by degenerate homogeneous structures of
linear type. In addition, any indecomposable simply-connected
Lorentzian symmetric space is isometric exactly to one of the
following: $(\R,-dt^2)$, the de Sitter or the anti de Sitter
space, or a Cahen-Wallach space.

In the Lorentz-K\"ahler case, according to \cite{Gal}, there is
only one possibility for a simply-connected, indecomposable (and
not irreducible), symmetric space of complex dimension $2$ and
signature $(2,2)$ with a null parallel complex vector field, that
is a manifold with holonomy
$$\f{hol}^{\gamma_1=0, \gamma_2=0}_{n=0} = \R
\left(\begin{array}{cc}i & i\\ -i & -i\end{array}\right)$$ in the
notation of \cite{Gal}. Note that this is the holonomy algebra in
Proposition \ref{proposition holonomy}. In order to get a plane
wave structure we add a plane wave front by considering a manifold
$(M,g,J)$ with holonomy $\f{hol}^{\gamma_1=0,
\gamma_2=0}_{n=0}\oplus\{0_n\}$. The result (see \cite{CL}) is up
to local holomorphic isometry, the space $\C^{n+2}$ with
\begin{equation}\label{metric galaev}
\bar{g}  = dw^1dz^1+dw^2dz^2+b(dw^1dw^1+dw^2dw^2)
+\sum_{a=1}^n(dx^adx^a+dy^ady^a),
\end{equation}
where the function $b$ only depends on $w^1$ and $w^2$ and
satisfies
$$\Delta b=R_0,\hspace{1em}
R_0\in\bb{R}-\{0\}.$$ This suggests to consider the
pseudo-K\"ahler manifold $(\C^{2+n},\bar{g})$ as a natural
Lorentz-K\"ahler analogue to Cahen-Wallach spaces. As
Cahen-Wallach spaces are simply-connected, in order to have an
actual analogue we only consider non-singular functions $b$, so
that $(\C^{2+n},\bar{g})$ is complete.

On the other hand, since Lorentzian singular scale-invariant
homogeneous plane waves are characterized by degenerate
pseudo-Riemannian homogeneous structures of linear type, from
Theorem \ref{theorem linear e-Kahler} the natural analogues to
these spaces are Lorentz-K\"ahler manifolds with degenerate
homogeneous pseudo-K\"ahler structures of linear type. More
precisely, the spaces $(\C^{n+2}-\{\| w\|_{\lambda}=0\},\bar{g})$
with $\|w\|_\lambda $ as in \eqref{norma}, $\bar{g}$ as in
\eqref{local metric general theorem} with $\e=-1$ and signature
$(2,2+2n)$, and $(\C^{n+2},\bar{g})$ with $\bar{g}$ given in
\eqref{metric galaev}, are composed by the twisted product of a
flat and totally geodesic complex manifold and a $2$-dimensional
complex manifold containing a null parallel complex vector field.
Moreover, the expression of \eqref{local metric general theorem}
and \eqref{metric galaev} are the same except for the function
$b$, which has a different Laplacian in each case. As a straight
forward computation shows, the curvature tensor of both metrics is
$$R=\frac{1}{2}\Delta b(dw^1\wedge dw^2)\otimes(dw^1\wedge
dw^2),$$ whence all the curvature information is contained in the
Laplacian of the function $b$. For this reason, analogously to
Lorentz plane waves, we call $\Delta b$ the profile of the metric.

It is worth noting that in the Lorentz case one goes from
Cahen-Wallach spaces to singular scale-invariant homogeneous plane
waves by making the profile be singular with a term $1/u^2$. Doing
so, the space is no longer geodesically complete and a
cosmological singularity at $\{u=0\}$ is created. In the same way,
in the Lorentz-K\"ahler case one goes from metric \eqref{metric
galaev} to \eqref{local metric general theorem} by making the
profile be singular with a term $1/\| w\|_{\lambda}^4$ and again
one transforms a geodesically complete space to a geodesically
uncomplete space, and a singularity at $\{\| w\|_{\lambda}=0\}$ is
created. This reinforce the parallelism and exhibits a close
relation between this two couples of spaces.

\bigskip

\begin{center}
\begin{tabular}{|c|c|c|}
\hline
\multirow{2}{*}{} & \T Symmetric  & \B Deg. homog. \\
 & space & of linear type\\[0.5ex]
\hline

\multirow{5}{*}{Lorentz} &  \T Cahen-Wallach & \T Singular s.-i.
homog.\\
       & spaces & plane wave \\
       & & \\
                & \footnotesize{Profile: $A(u)=B(const.)$} & \footnotesize{Profile: $A(u)=B/u^2$}\\

                & \footnotesize{Geodesically complete} & \footnotesize{Geodesically
                uncomplete}\\
                                [0.5ex]
\hline
\multirow{5}{*}{Lorentz-K\"ahler} & \T $\C^{2+n}$ with metric  & \T $\C^{2+n}-\{\| w\|_{\lambda}=0\}$\\ & \eqref{metric galaev} & with metric \eqref{local metric general theorem}\\
 & & \\
  & \footnotesize{Profile: $\Delta b=R_0(const.)$} & \footnotesize{Profile: $\Delta b=R_0/\| w\|_{\lambda}^4$}\\
 & \footnotesize{Geodesically complete} & \footnotesize{Geodesically
 uncomplete}\\
  [0.5ex]
\hline
\end{tabular}
\end{center}

Finally, note that a pseudo-quaternion K\"ahler manifold admitting
a degenerate homogeneous pseudo-quaternion K\"ahler structure of
linear type must be flat, suggesting that the notion of
homogeneous plane wave cannot be adapted to pseudo-quaternion
K\"ahler geometry.

\section*{\small{Acknowledgements}}

\footnotesize{The authors are deeply indebted to Prof. Andrew
Swann and Prof. P.M. Gadea for useful conversations about the
topics of this paper.

\noindent This work has been partially funded by MINECO (Spain)
under project MTM2011-22528.}

\end{document}